\documentclass{siamltex}
\usepackage{multirow}
\usepackage[english]{babel}
\usepackage{amsmath}
\usepackage{amsfonts}
\usepackage{amssymb}
\usepackage{graphicx}
 \usepackage{array,multirow,makecell}
\usepackage{tikz}
\usepackage{here}
\usepackage{yhmath}
\usepackage{mathdots}
\usepackage{MnSymbol}
\usepackage{algorithm}
\usepackage{algorithmic}
\usepackage{graphicx}
\usepackage{amsfonts}            
\usepackage{amsmath}
\usepackage{here}
\usetikzlibrary{matrix,calc}
\usepackage{caption}
\newtheorem{prop}{Proposition}

\newtheorem{lemm}{Lemma}

\newtheorem{coro}{Corollary}

\newtheorem{remark}{Remark}

\title{Generalized  $\mathcal{L}$-product for hight order tensors with applications using GPU computations.}

\author{ A.H Bentbib\thanks{Facult\'e des Sciences et Techniques-Gueliz, Laboratoire de Math\'ematiques Appliqu\'ees et Informatique, Marrakech, Morocco} \and
M. Elalj\footnotemark[4]
	\and A. El Hachimi\footnotemark[4] 
	\and K. Jbilou\footnotemark[4] \thanks{Universit\'e du Littoral Cote d'Opale, LMPA, 50 rue F. Buisson, 62228 Calais-Cedex, France }
	\and {A. Ratnani  \thanks{laboratory MSDA, Mohammed VI Polytechnic University, Green City, Morocco}.}}

\begin{document}
	
	\maketitle

	\begin{abstract}
		In this paper, we will present a generalization of the $\mathcal{L}$-tensor product ($*_\mathcal{L}$-product) including generalization of the well known tensor cosine and T-products that were defined for third-order tensors and based on fast Fourier transform and discrete cosine transform (DCT). We will give some applications on tensor completion. To solve some optimization problems linked with  the problem of tensor completion, we will use the Proximal Gradient Algorithm (PGA) to solve some derived optimization problems. Numerical tests are given to show the effectiveness of the proposed methods and also present some tests using GPU computation.
	\end{abstract}
	
	
	\section{Introduction}
	In the last decade, tensors become an important multilinear algebra tool involved in many modern problems such completion \cite{elha1,ding2019low,ji2016tensor}, principal component analysis \cite{lu2019tensor}, image processing \cite{elguide,kilmer2011factorization,kolda} and others. \\The classical $n$-mode product leads to many concepts and developements when working with multidimensional data. The CP and the Tucker compressions were introduced as natural generalization of the classical singular value decomposition (SVD) for matrices; see \cite{golubvanloan,kolda,kilmer2011factorization,lu2019tensor,ji2016tensor}.  \\
	In the last years,  new tensor-tensor products such as cosine-product (c-product), using discrete cosine or T-product, using Fast Fourier Transform (FFT),  were introduced for third-order tensors, studied and applied to image processing and other fields; see  \cite{elichi,zhang2014novel,kilmer2,lu2019tensor,elguide}. In the present paper, we generalize those tensor-tensor products for high-order tensors. Using those new products, we will propose new completion models. We give some theoretical results and some numerical examples in color video processing. \\
	The outline of this paper will as follows
	In Section \ref{sec1} we will give some definitions  and remind some known results of the third-order tensor-tensor product based on Fast Fourier Transform and cosine transform. 
	 In Section \ref{sec2}, we present our new generalized  c-product  for any order of tensors and give some important results.
	 Section \ref{sec3} presents a novel models of tensor completion for  tensors of any order by using the PGA. Section \ref{sec5}  will be devoted to some numerical experiments with  someexperiments using GPU. 

	\section{Definitions and notations}
		\label{sec1}
		In this subsection we will define some notions that will help us in the paper. We  denote   tensors by Euler script letters, e.g., $\mathcal{X}$, matrices will be denoted by boldface capital letters ,e.g., $X$, vectors by boldface lowercase letters e.g., $\textbf{x}$ and scalars by lowercase letters, e.g., $x$. Also we will denote the $(i_1,\, i_2,\, ...,\, i_K)^{th}$ for a $K^{th}$-order tensor $\mathcal{X}$ by $\mathcal{X}_{i_1,\, i_2,\, ...,\, i_K}$. Also we will denote $\mathbb{C}^{I_1 \times I_2 \times \dots \times I_K}$ by $\mathbb{K}^{I_1 \times I_2}_{I_3 \times \dots \times I_K}$, and the $K^{th}$-order tensor-scalar space as $\mathbb{K}_{I_3 \times \dots \times I_K}=\mathbb{C}^{1\times 1 \times I_3 \times \dots \times I_K}$.\\
		Let $\mathcal{A}$ and $\mathcal{B}$ be  two $K^{th}$-order tensors in $\mathbb{K}^{I_1 \times I_2}_{I_3\times \dots \times I_k}$, we will define the inner product between $\mathcal{A}$ and $\mathcal{B}$ by 
		\begin{equation}
		\left<\mathcal{A},\, \mathcal{B}\right>=\sum_{i_1,\, i_2,\, \dots ,\, i_K=1}^{I_1,\, I_2,\, \dots ,\, I_K}\mathcal{A}_{i_1,\, i_2,\, \dots,\, i_K}\mathcal{B}_{i_1,\, i_2,\, \dots,\, i_K}.
		\end{equation}
		The associated norm  is defined by 
		\begin{equation}
		\left\Vert \mathcal{A}\right\Vert_F^2=\sum_{i_1,\, i_2,\, \dots ,\, i_K=1}^{I_1,\, I_2,\, \dots ,\, I_K}\mathcal{A}_{i_1,\, i_2,\, \dots,\, i_K}^2.
		\end{equation}
		The notion of columns and rows of matrices are extended to the case of tensors, where in this case we talk about $n$-mode fiber instead of columns and rows, with the $n$-mode fiber of a $K^{th}$-order tensor $\mathcal{X}$ is defined by fixing all the indexes except the $n^{th}$ one.\\
		There is some ways to transform a tensor to a matrix which consist to make the operations on tensors easier, there is for example the $n$-mode matricization \cite{kolda,ji2016tensor} defined as follows
		\begin{definition}
			Let $\mathcal{X}\in \mathbb{K}^{I_1 \times I_2}_{I_3 \times \dots \times I_N}$, then the $n$-mode matricization of $\mathcal{X}$ denoted by $X_{(n)}\in \mathbb{K}^{I_n \times \prod_{k=1,\, k\neq n}^{N}I_k}$ and it is defined by making the $n$-mode fibers as columns of $X_{(n)}$, i.e., the $(i_1,\, i_2,\, \dots  , \, i_N)^{th}$ element of $\mathcal{X}$ maps to a matrix element $(i_n,\, j)$ satisfying 
			$$j=1 + \sum_{k=1,\, j\neq n}^{N}(i_k-1)J_k\,\, for \,\, J_k=\prod_{m=1,\, m\neq n}^{N}I_m.$$
		\end{definition}
		For third-order tensors $\mathcal{X}\in \mathbb{K}^{I_1 \times I_2}_{I_3}$, $X_{(1)}$, $X_{(2)}$ and $X_{(3)}$ are given by 
		\begin{eqnarray*}
		X_{(1)} &=& \left[\mathcal{X}^{(1)},\, \mathcal{X}^{(2)},\, \dots , \, \mathcal{X}^{(I_3)}  \right],\\
		X_{(2)}&=&\left[\left(\mathcal{X}^{(1)}\right)^T,\, \left(\mathcal{X}^{(2)}\right)^T,\, \dots , \, \left(\mathcal{X}^{(I_3)}\right)^T  \right],\\
		X_{(3)} &=&\left[{\tt sq}\left(\mathcal{X}(:,\,1,\, :)\right)^T,\, {\tt sq}\left(\mathcal{X}(:,\,2,\, :)\right)^T,\, \dots , \, {\tt sq}\left(\mathcal{X}(:,\,I_2,\, :)\right)^T   \right],
		\end{eqnarray*}
		where $\mathcal{X}^{(i)}$ denotes the $i^{th}$ frontal slice $\left(\mathcal{X}^{(i)}=\mathcal{X}(:,:,i)\right)$ and ${\tt sq}$ transforms the tensor $\mathcal{X}\in \mathbb{K}^{I_1 \times 1}_{I_3}$ to a matrix $X\in \mathbb{K}^{I_1 \times I_3}$, i.e., $X={\tt sq}\left(\mathcal{X}\right)$.\\
		The $n$-mode product, which is a product between a tensor and a matrix in the $n$-mode \cite{kolda} is defined in  the following definition
		\begin{definition}
			Let $\mathcal{X}\in \mathbb{K}^{I_1 \times I_2}_{I_3 \times \dots \times I_K}$ and $U \in \mathbb{K}^{J\times I_n}$ where  $J$ is a positive nonzero integer. Then the $n$-mode product $\mathcal{X}\times_n U$ is the tensor  in $\mathbb{K}^{I_1 \times I_2}_{I_3 \times \dots \times I_{n-1} \times J \times I_{n+1} \times \dots \times I_K}$, where its $\left(i_1,\dots, i_{n-1}, j, i_{n+1}, \dots , i_K\right)^{th}$  element is defined by
			\begin{equation}
			\left(\mathcal{X}\times_n U\right)_{i_1,\, \dots,\, i_{n-1},\, j ,\, i_{n+1},\, \dots, \, i_K}=\sum_{i_n=1}^{I_n}\mathcal{X}_{i_1,\, i_2,\, \dots ,\, i_K}U_{j,\,i_n}.
			\end{equation}
		\end{definition}
	Some useful properties of the $n$-mode product are given as follows.  Let the tensor  $\mathcal{X}$, and the matrices $U$ and $V$ of appropriate sizes, then  
	\begin{eqnarray*}
		\mathcal{Y}=\mathcal{X}\times_n U & \Longleftrightarrow & Y_{(n)}=UX_{(n)}, \; and \\
		 \mathcal{X}\times_n U \times_m V & = & \mathcal{X}\times_m V \times_n U.
		\end{eqnarray*}
We will also use the notion of tensor face-wise product  defined next.
		\begin{definition}
			Let $\mathcal{X}\in \mathbb{K}^{I_1 \times n}_{I_3}$ and $\mathcal{Y}\in \mathbb{K}^{n \times I_2}_{I_3}$ two third-order tensors, then the face-wise product between $\mathcal{X}$ and $\mathcal{Y}$ is given by the tensor of size $I_1 \times I_2 \times I_3$ where its $i^{th}$ frontal slice is given from the product between the $i^{th}$ frontal slices of $\mathcal{X}$ and $\mathcal{Y}$, i.e.,
			\begin{equation}
			\left(\mathcal{X}\triangle \mathcal{Y}\right)^{(i)}=\mathcal{X}^{(i)}\triangle \mathcal{Y}^{(i)}.
			\end{equation} 
		\end{definition}
	Classical tensor decompositions such as  CP decomposition \cite{kolda}, Tucker decomposition \cite{ji2016tensor}, block term decomposition \cite{block-terme decomposition} give nice  results in many tensor applications. However, those  decompositions suffer from the high computational cost for large problems. In the recent years new  tensor decompositions of the third-order case and  based on tensor-tensor product using the Fourier domain  such as the t-product \cite{kilmer2011factorization} and cosine-product (c-product) \cite{kernfeld2015tensor}, have been defined and used for many image processing applications; see   \cite{zhang2014novel,kilmer2021tensor,elha1,lu2019tensor}.\\ In this section we will try to remind the most important results of those types of tensor-tensor product.\\
		The main idea of this type of tensor products is to transform the tensors to another domain which is called the transform domain, like Fourier domain, cosine domain. Then  all the operations are done in the transformed domain using for examlple FFT on each tube to speed-up the executing time. This kind of transformation could be defined in the following.
		\begin{definition}
			\label{def 5}
			Let  $M$ be an invertible matrix of size $I_3 \times I_3$, we define the operator $	\mathcal{L}$  as 
			\begin{eqnarray*}
				\mathcal{L}:\; \mathbb{K}^{I_1 \times I_2}_{I_3}&  \longrightarrow &\mathbb{K}^{I_1 \times I_2}_{I_3} \\
				\mathcal{A}  &\longrightarrow &\mathcal{A}\times_3 M
			\end{eqnarray*}
		 and its inverse is defined as
			\begin{eqnarray*}
				\mathcal{L}^{-1}: \; \mathbb{K}^{I_1 \times I_2}_{I_3} & \longrightarrow & \mathbb{K}^{I_1 \times I_2}_{I_3} \\
			\mathcal{A} &  \longrightarrow & \mathcal{A}\times_3 M^{-1}
			\end{eqnarray*}
		\end{definition}
		Now we can define the tensor-tensor product of two third-order tensors.
		\begin{definition}
			Let $\mathcal{L}$ be an invertible operator, then the tensor-tensor product between two third-order tensors associated with  the operator $\mathcal{L}$, is denoted by $*_{\mathcal{L}}$ and is given by 
			\begin{equation}\label{eq 2.5}
			\mathcal{A}*_{\mathcal{L}}\mathcal{B}=\mathcal{L}^{-1}\left(\mathcal{L}\left(\mathcal{A}\right)\triangle \mathcal{L}\left(\mathcal{B}\right)\right).
			\end{equation}
	where the tensors $\mathcal{A}$ in $\mathbb{K}^{I_1 \times l}_{I_3}$ and   $\mathcal{B}$ in $\mathbb{K}^{l \times I_2}_{I_3}$.
		\end{definition}
	
	\noindent 	The matrix $M$ in Definition  \ref{def 5} depends on the type of the product, for example if we use the t-product \cite{kilmer2011factorization}, the matrix $M$ is the matrix of discrete Fourier transform $F_{I_3}$ where the Fourier matrix $F_n  \in \mathbb{C}^{n\times n}$ is given by
		\begin{equation}
		F_n=[\omega_n^{(i-1)(j-1)}];\, i,j=1,\ldots,n-1; \; and \; \omega_n=e^{\frac{-2i \pi}{n}}.
				\end{equation}
%
		for $n\in \mathbb{N}^*$. Notice that $\dfrac{F_n}{\sqrt{n}}$ is unitary, i.e., $F_n F_n^*=nI_n$.\\
		In the case of c-product \cite{kernfeld2015tensor}, the matrix $M$ is defined as 
		\begin{equation}\label{m}
		M=W^{-1}_{I_3}C_{I_3}\left(I_{I_3}+Z_{I_3}\right),
		\end{equation}
		where $W_{I_3}={\tt diag}\left(C_{I_3}(:,1)\right)$, the matrix $Z_{I_3}$ is the circulant upshift matrix defined by  $Z_{I_3}={\tt diag(ones(I_3-1,1),1)}$ and $C_{I_3}$ is the matrix of discrete cosine transform of size $I_3 \times I_3$ and its $(i,j)^{th}$ element is defined as 
		\begin{equation}
		\left(C_{I_3}\right)_{i,\,j}=\sqrt{\dfrac{2-\delta_{i,\,j}}{I_3}} \cos\left(\dfrac{(i-1)(2j-1)\pi}{2I_3}\right) \;\; with \;\; 1\leq i,\, j\leq I_3, 
		\end{equation}
		where $\delta_{i,\,j}$ is the Kronecker symbol. Notice that the matrix $C_{n}$ is orthogonal for all $n\in \mathbb{N}^*$. We have also to mention that, in this case also, the matrix $M$  is invertible and $M^{-1}_{I_3}=\left(I_{I_3}+Z_{I_3}\right)^{-1}C^*_{I_3}W_{I_3}.$
		Using those  tensor-tensor products, all the classical matrix decomposition,  such as svd, QR and Shur decompositions have been generalized to the tensor case; see \cite{kernfeld2015tensor,kilmer2013third}.
		Many applications of tensor-tensor product use some optimization algorithms  and in our present work, we will use the   Proximal Gradient Algorithm  \cite{elha2} in tensor completion.  The method consists in solving the optimization problem 
		\begin{eqnarray}
		&\underset{\mathcal{X}\in \mathbb{H}}{\min}&\, g(\mathcal{X}) \nonumber \\
		&{\tt s.t}& \textbf{A}\left(\mathcal{X}\right)=\mathcal{B},
		\end{eqnarray} 
		where $\mathbb{H}$ is an Hilbert space equipped with a norm $\left\Vert . \right\Vert$, $g$ is a continuous function, $\textbf{A}$ is a linear map and $\mathcal{B}$ is an observation.\\
		By referring to \cite{beck2009fast,lin2009fast}, this optimization problem can be solved by solving the following one
		\begin{equation}\label{eq11}
		\underset{\mathcal{X}\in \mathbb{H}}{\min}\, \{F(\mathcal{X})=\mu g(\mathcal{X})+f(\mathcal{X})\},
		\end{equation}
		where $f(\mathcal{X})=\dfrac{1}{2}\left\Vert \mathcal{A}(\mathcal{X})-\mathcal{B} \right\Vert^2$ and $\mu >0$ is the relaxation parameter. The penality function $f$ is convex and smooth with Lipshitz continuous gradient, with Lipshitz constant $l_f$. 
	To solve  \eqref{eq11}, we minimize  the  quadratic function   $Q(\mathcal{X},\mathcal{Y})$, where $\mathcal{Y}$ is chosen  and $Q$ is defined as follows
		\begin{equation}\label{eq 12}
		Q\left(\mathcal{X},\mathcal{Y}\right)=\mu g(\mathcal{X})+f(\mathcal{Y})+ \left<\nabla f(\mathcal{Y}),\, \mathcal{X}-\mathcal{Y}\right>+\dfrac{l_f}{2}\left\Vert \mathcal{X}-\mathcal{Y}\right\Vert^2.
		\end{equation}
Getting a solution  $\mathcal{X}$  satisfying \eqref{eq 12} is equivalent to solve  the following minimization problem 
		\begin{equation} \label{eq 13}
		\underset{\mathcal{X}\in \mathbb{H}}{\min}\, Q(\mathcal{X},\mathcal{Y})=\underset{\mathcal{X}\in \mathbb{H}}{\min}\, \mu g(\mathcal{X})+\dfrac{l_f}{2}\left\Vert \mathcal{X}-\mathcal{G}\right\Vert^2,
		\end{equation}
		where $\mathcal{G}=\mathcal{Y}-\dfrac{1}{l_f}\nabla f(\mathcal{Y})$.  
		The problem is solved  iteratively by computing $\mathcal{X}^{p+1}$ such that 
		\begin{equation}
		\mathcal{X}^{p+1}=\underset{\mathcal{X}\in \mathbb{H}}{\arg\, \min}\, Q(\mathcal{X},\, \mathcal{Y}^p),
		\end{equation}
	In  \cite{lin2009fast},  $\mathcal{Y}^p$ was computed by  $\mathcal{Y}^p=\mathcal{X}^p+\dfrac{t_{p+1}-1}{t_p}\left(\mathcal{X}^p-\mathcal{X}^{p-1}\right)$ instead of $\mathcal{Y}^p=\mathcal{X}^p$ for computationally reasons  
	 and $t_{p+1}=\dfrac{1+\sqrt{4t_p^2+1}}{2}$. The steeps of this algorithm can be summarized in the following algorithm 
		\begin{algorithm}[H]
			\caption{Proximal Gradient Algorithm (PGA).}
			\label{alg:1}
			\begin{algorithmic}[1]
				\WHILE {not converged}
				\STATE $\mathcal{Y}^p=\mathcal{X}^p+\dfrac{t_{p-1}-1}{t_p}\left(\mathcal{X}^p-\mathcal{X}^{p-1}\right)$.
				\STATE $\mathcal{G}^p=\mathcal{Y}^p-\dfrac{1}{l_f}\nabla f\left(\mathcal{Y}^p\right)$.
				\STATE $\mathcal{X}^{p+1}=\underset{\mathcal{X}\in \mathbb{H}}{\arg\, \min}\, \mu g(\mathcal{X})+\dfrac{l_f}{2}\left\Vert \mathcal{X}-\mathcal{G}^p\right\Vert^2$.
				\STATE $t_{p+1}=\dfrac{1+\sqrt{4t_p^2+1}}{2}$.
				\STATE $p=p+1.$
				\ENDWHILE
			\end{algorithmic}
		\end{algorithm}
	
	\begin{section}{Generalized tensor-tensor  cosine product}
		\label{sec2}
		The main inconvenient  of tensor-tensor products above is the fact that they  could be used only for third-order tensors. In \cite{martin2013order}, the authors proposed a generalization of the t-product and  in the present work we propse   the generalization of the c-product with some applications.	We will first recall some important results linked with the c-product for third-order tensors described in \cite{kernfeld2015tensor}. For two third-order tensors $\mathcal{A}\in \mathbb{K}^{I_1 \times l}_{I_3}$ and $\mathcal{B}\in \mathbb{K}^{l \times I_2}_{I_3}$, the c-product $	\mathcal{A}*_c \mathcal{B}$  is defined by
		\begin{equation}
			\mathcal{A}*_c \mathcal{B}= {\tt ten}\left({\tt btph}\left(\mathcal{A}\right){\tt btph}\left(\mathcal{B}\right)\right) \in \mathbb{K}^{I_1 \times I_2}_{I_3},
		\end{equation}
		where ${\tt btph}$ represents the block-Toeplitz-plus-Hankel matrix   defined as
		\begin{equation}\label{eq 3.22}
			{\tt btph}\left(\mathcal{A}\right)=\begin{pmatrix}
			\mathcal{A}^{(1)} &  \dots  &  \dots   &   \mathcal{A}^{(I_3)}\\
			\mathcal{A}^{(2)} &  \mathcal{A}^{(1)}  &  \dots   &   \mathcal{A}^{(I_3-1)}\\
			 \vdots & \ddots   &   \ddots   &  \vdots \\
			\mathcal{A}^{(I_3)} & \dots   &  \dots   &   \mathcal{A}^{(1)}
			\end{pmatrix} + 
			\begin{pmatrix}
			\mathcal{A}^{(2)} &  \dots  &   \mathcal{A}^{(I_3)} & \textbf{0}\\
			\vdots &\reflectbox{$\ddots$}   &  \reflectbox{$\ddots$}  &   \mathcal{A}^{(I_3)}\\
			\mathcal{A}^{(I_3)} & \textbf{0}   &  \reflectbox{$\ddots$}  &  \vdots \\
		\textbf{0} & \mathcal{A}^{(I_3)}    &  \dots   &   \mathcal{A}^{(2)}
			\end{pmatrix} \in \mathbb{K}^{I_1 I_3 \times  I_2 l}
				\end{equation}
				and the operator ${\tt ten}$ is the operator to get back a tensor from its ${\tt btph}$, where ${\tt ten}\left({\tt btph}\left(\mathcal{A}\right)\right)=\mathcal{A}.$ Computing this product by the above formula can be expensive, because the matrix ${\tt btph}$  may  be very large. Some properties of ${\tt btph}$ are given in \cite{kernfeld2015tensor}. Among them, the fact that for a tensor $\mathcal{A}\in \mathbb{K}^{I_1 \times I_2}_{I_3}$,  the matrix ${\tt btph}$  is block diagonalizable by  $\left(C_{I_3} \otimes I_{I_1}\right)$ and  we have 
				\begin{equation}\label{eq 3.33}
					\left(C_{I_3} \otimes I_{I_1}\right){\tt btph}\left(\mathcal{A}\right)\left(C_{I_3}^* \otimes I_{I_2}\right)={\tt bdiag}\left(\hat{\mathcal{A}}\right),
				\end{equation}
				where $\hat{\mathcal{A}}=\mathcal{A}\times_3 M$ with  $M$ is defined by \eqref{m}  and 
				\begin{equation}
					{\tt bdiag}\left(\hat{\mathcal{A}}\right)=\begin{pmatrix}
					\hat{\mathcal{A}}^{(1)} &  &   &   \\
					& \hat{\mathcal{A}}^{(2)} &   &  \\
					&   & \ddots  &    \\
					&  &   & \hat{\mathcal{A}}^{(I_3)}
					\end{pmatrix}\in \mathbb{K}^{I_1 I_3 \times I_2 I_3}. 
				\end{equation}
				From this last result,   we can define the c-product between two third-order tensors $\mathcal{A}$ and $\mathcal{B}$ of appropriate sizes with $I_3$ frontal slices. 
				Before giving the generalized version of the high-order c-product, we give some definitions and notations.
				First, we will call a tensor in the scalar space ($\mathbb{K}_{I_3 \times \dots \times I_N}$) a scalar-tensor, which will replace the notion of tubes in the case of third-order tensors.
		For an $N^{th}$-order tensor $\mathcal{X}\in \mathbb{K}^{I_1 \times I_2}_{I_3 \times \dots \times I_N}$,  we will define the  operator $\tt Vec$ that transforms the tensor into a matrix as follows 
			\begin{equation}\label{eq 3.5}
			{\tt Vec(\mathcal{X})}=[X^1,\, X^2,\, ...,\, X^{n_3}, \, X^{n_3+1},\, ...,\, X^{P}]
			\end{equation} 
			where $P=I_3I_4\dots I_N$ and $X^p=\mathcal{X}(:,:,k_3,k_4, \dots , k_N)\in \mathbb{K}^{I_1 \times I_2}$ with $p=k_3+\displaystyle \sum_{i=4}^{N} \left (\dfrac{\left(k_i-1\right) P}{\prod_{s=i}^{N}I_i} \right ).$ This generalizes  the notion of frontal slices in the case of third-order tensors. Notice that in the case of scalar-tensors $\mathcal{X}\in \mathbb{K}^{1 \times 1}_{I_3 \times \dots \times I_N}$, the  $X^p$'s are  scalars and $	{\tt Vec(\mathcal{X})}$ is a row vector.\\
 In \cite{kernfeld2015tensor}, the  authors  defined the block-Toeplitz-plus-Hankel matrix for third-order tensor and  here we will define the block-Toelitz-plus-Hankel matrix for an $N^{th}$-order tensor for $N\geq3$ by using the block Toeplitz plus Hankel for an $(N-1)^{th}$-order tensor. To explain this we will present this procedure  only  for a fourth-order tensor $\mathcal{X}\in \mathbb{K}^{I_1 \times I_2}_{I_3 \times I_4}$. We define the matrix ${\tt btph} \in \mathbb{R}^{n_1P \times n_2P}$ as the matrix block-Toeplitz-plus-Hankel where each block of the matrices Toeplitz and Hankel is a block-Toeplitz-plus-Hankel matrix of third-order tensors, respectively. Then the block-Toeplitz-plus-Hankel matrix of $\mathcal{X}$ is defined as follows
			\begin{eqnarray*}\label{eq16}
			{\tt btph}(\mathcal{X})&=&\begin{pmatrix}
			{\tt btph}\left(\mathcal{X}(:,:,:,1) \right) & {\tt btph}\left(\mathcal{X}(:,:,:,2) \right) & \dots & \dots & {\tt btph}\left(\mathcal{X}(:,:,:,I_4) \right) \\
			{\tt btph}\left(\mathcal{X}(:,:,:,2) \right) & {\tt btph}\left(\mathcal{X}(:,:,:,1) \right) & \dots  & \dots  & {\tt btph}\left(\mathcal{X}(:,:,:,I_4-1) \right) \\
			\vdots & \vdots & \vdots & \vdots  & \vdots\\
			{\tt btph}\left(\mathcal{X}(:,:,:,I_4) \right) & {\tt btph}\left(\mathcal{X}(:,:,:,I_4 -1) \right) & \dots  & \dots & {\tt btph}\left(\mathcal{X}(:,:,:,1) \right)
			\end{pmatrix} \nonumber \\
		  &&\\
			&+&
			\begin{pmatrix}
			{\tt btph}\left(\mathcal{X}(:,:,:,2) \right) & \dots & {\tt btph}\left(\mathcal{X}(:,:,:,I_4) \right)  & \textbf{0} \\
			\vdots &\reflectbox{$\vdots$}   &  \reflectbox{$\vdots$}  &   {\tt btph}\left(\mathcal{X}(:,:,:,I_4) \right)\\
			{\tt btph}\left(\mathcal{X}(:,:,:,I_4)\right) & \textbf{0}& \reflectbox{$\vdots$} & \vdots \\
			\textbf{0} & {\tt btph}\left(\mathcal{X}(:,:,:,I_4)\right) & \dots & {\tt btph}\left(\mathcal{X}(:,:,:,2)\right)
			\end{pmatrix}\hspace{-2cm}.
			\end{eqnarray*}		
 ${\tt ten}({\tt btph}(\mathcal{X}))=\mathcal{X}$, this operator allows to reconstruct the original tensor from its associate ${\tt btph}$ matrix.
The  block diagonal matrix of the tensor $\mathcal{X}$ is given by 
			\begin{equation}\label{bd}
			{\tt bdiag}(\mathcal{X})=\begin{pmatrix}
			X^1 & \textbf{0} & \dots & \dots & \textbf{0} \\
			\textbf{0} & X^2 & \textbf{0} & \dots & \textbf{0} \\
			\vdots & \ddots & \ddots & \ddots & \vdots \\
			\vdots & \ddots & \ddots & \ddots & \vdots \\
			\textbf{0} & \dots & \dots & \dots & X^P
			\end{pmatrix}\in \mathbb{R}^{n_1P \times n_2P},
			\end{equation}
where the   $X^i$ 's matrices of size $I_1 \times I_2$ are given  in \eqref{eq 3.5}   as  the representative matrices of $\mathcal{X}$. Notice that in the case of scalar-tensors, the blocks $X^i$  are scalars in $\mathbb{C}$  and  in this case, the block diagonal matrix \eqref{bd} is just a diagonal matrix and  ${\tt bdiag}(\mathcal{X})={\tt diag}(\mathcal{X})$.

		\begin{subsection}{Generalized cosine product c-product}
	In this subsection, we will introduce a generalized version of 	 the c-product for high-order tensors. To this end, we first need some theoretical results. First, remind that in the third-order case, the Toeplitz-plus-Hankel matrix  of a  tube $\textbf{a}\in \mathbb{K}_{I_3}$ is diagonalizable using the DCT matrix  of order $I_3$, i.e.,
			\begin{equation}\label{d}
			C_{I_3}{\tt tph}(\textbf{a})C_{I_3}^*={\tt diag}(d); \; where \; d=W_{I_3}^{-1}C_{I_3}\left(I+Z_{I_3}\right){\tt vec}(\textbf{a}),
			\end{equation}
			with  $W_{I_3}$ and $Z_{I_3}$ as defined earlier and ${\tt vec}(\textbf{a})$ is the vector of size $I_3$ whose elements are the coefficients of the tube $\textbf{a}$ . This result is extended to the high-order case by considering scalar-tensors instead of tubes.

			\begin{theorem}\label{theo 8}
				Let $\mathcal{X}\in \mathbb{K}_{I_3\times \dots \times I_N}$,  a scalar-tensor, then its block-Toeplitz-plus-Hankel matrix is diagonalizabl,  and we have
				\begin{equation} \label{eq 19}
				\left(C_{I_N}\otimes C_{I_{N-1}} \otimes \dots \otimes C_{I_3}  \right) {\tt btph}\left(\mathcal{X}\right)\left(C_{I_N}^*\otimes C_{I_{N-1}}^* \otimes \dots \otimes C_{I_3}^*  \right)={\tt diag}\left(\hat{\mathcal{X}}\right)
				\end{equation}
			where $\hat{\mathcal{X}}=\mathcal{X}\times_3 M_{I_3}\times_4 M_{I_4} \times_5 \dots \times_N M_{I_N}$ and $C_{I_N}$ are DCT matrices.
			\end{theorem}
		\medskip
		
			\begin{proof}
			For simplicity we  consider only the case of  fourth-order scalar-tensors  $\mathcal{X}$ in $\mathbb{K}^{1\times 1}_{I_3 \times I_4}$. Using the fact that the matrix block-Toeplitz-plus-Hankel for a third-order tensor is block-diagonalizable using the discrete cosine matrix, we get 
				\begin{eqnarray}\label{eq 3.99}
					\left(I_{I_4}\otimes C_{I_3}\right){\tt btph}\left(\mathcal{X}\right)\left(I_{I_4}\otimes C_{I_3}^*\right)={\tt btph}\left(\mathcal{X}'\right),
				\end{eqnarray}
				with $\mathcal{X}'$ is the tensor such that  $\mathcal{X}'(:,:,:,k_4)=C_{I_3}{\tt btph}(\mathcal{X}(:,:,:,k_4))C_{I_3}^*$ for $k_4=1, \dots,\, I_4$. Hence  $\mathcal{X}'=\mathcal{X}\times_3 M_{I_3}$  and 
				\begin{eqnarray*}
			{\tt btph}\left(\mathcal{X}'\right)&=&\begin{pmatrix}
				diag(\mathcal{X}'_{1}) &  \dots  &  \dots   &  diag(\mathcal{X}'^{I_4})\\
				diag(\mathcal{X}'_{2}) &  diag(\mathcal{X}'_{1}) &  \dots   &   diag(\mathcal{X}'_{I_4-1})\\
				\vdots & \ddots   &   \ddots   &  \vdots \\
				diag(\mathcal{X}'_{I_4}) & \dots   &  \dots   &  diag(\mathcal{X}'_{1})
				\end{pmatrix}\\
				&+& 
				\begin{pmatrix}
				diag(\mathcal{X}'_{2})&  \dots  &   diag(\mathcal{X}'_{I_4}) & \textbf{0}\\
				\vdots &\reflectbox{$\ddots$}   &  \reflectbox{$\ddots$}  &   diag(\mathcal{X}'_{I_4})\\
				diag(\mathcal{X}'_{I_4}) & \textbf{0}   &  \reflectbox{$\ddots$}  &  \vdots \\
				\textbf{0} & diag(\mathcal{X}'_{I_4})   &  \dots   &  diag(\mathcal{X}'_{2})
				\end{pmatrix}
				\end{eqnarray*}
				where $\mathcal{X}'_i= \mathcal{X}'(:,:,:,i)$ for $1=1, 2, \dots, I_4$. Using a similar result as \eqref{eq 3.33},  we get  
				\begin{equation}\label{eq 3.100}
				 \left(C_{I_4}\otimes I_{I_3}\right){\tt btph}\left(\mathcal{X}'\right)\left(C_{I_4}^*\otimes I_{I_3}\right)={\tt diag}\left(d\right) \in \mathbb{K}^{I_3I_4 \times I_3I_4}
				\end{equation}
				Then, from \eqref{eq 3.99} and \eqref{eq 3.100} we obtain 
				\begin{eqnarray}\label{bt1}
					\left(C_{I_4}\otimes I_{I_3}\right)\left(I_{I_4}\otimes C_{I_3}\right){\tt btph}\left(\mathcal{X}\right)\left(I_{I_4}\otimes C_{I_3}^*\right)\left(C_{I_4}^*\otimes I_{I_3}\right)={\tt diag}\left(d\right).
				\end{eqnarray}
				Therefore, we can deduce that  ${\tt diag}(d)={\tt bdiag}(\hat{\mathcal{X}})={\tt diag}\left(\hat{\mathcal{X}}\right)$, with $\hat{\mathcal{X}}=\mathcal{X}\times_3 M_{I_3} \times M_{I_4}$ is a scalar-tensor, with $M_{I_i}=W_{I_i}^{-1}C_{I_i}\left(I_i+Z_{I_i}\right)$ for $i=3, 4$. 
				Then, from \eqref{bt1}, we obtain 
				\begin{eqnarray*}
					\left(C_{I_4}\otimes C_{I_3}\right){\tt btph}\left(\mathcal{X}\right)\left(C_{I_4}^*\otimes C_{I_3}^*\right)={\tt bdiag}\left(\hat{\mathcal{X}}\right).
				\end{eqnarray*}
			The reslut for fifth-order tensor could be obtained from the one for the fourth-order case and so on, 
			the result for an $N^{th}$-order scalar-tensor,  can be \textcolor{brown}{found} recursively
			  from the one for an  $(N-1)^{th}$-order scalar-tensor. This is due to the fact that the $\tt btph$ matrix of an $N^{th}$-order tensor is comoputed by using the $\tt btph$ matrix of an $(N-1)^{th}$-order tensor. Therefore, for a general scalar-tensor order $\mathcal{X} \in \mathbb{K}^{1\times 1}_{I_3 \times \dots \times_4 I_N}$, we obtain
				\begin{equation*}
				\left(C_{I_N}\otimes C_{I_{N-1}} \otimes \dots \otimes C_{I_3}\right){\tt btph}(\mathcal{X})\left(C_{I_N}^*\otimes C_{I_{N-1}}^* \otimes \dots \otimes C_{I_3}^*\right)={\tt diag}(\hat{\mathcal{X}}),
				\end{equation*}
			where  $\hat{\mathcal{X}}=\mathcal{X}\times_3 M_{I_3} \times_{4} M_{I_4} \times_5 \dots \times M_{I_N}$.
			\end{proof}
		
		\medskip 
			\noindent  Next, we define the new operator $\mathcal{L}$ as follows.
			\begin{definition}
				Let $\mathcal{L}: \mathbb{K}^{I_1 \times I_2}_{I_3 \times \dots \times I_N}\longrightarrow \mathbb{K}^{I_1 \times I_2}_{I_3 \times \dots \times I_N} $ be the operator  defined by $$\mathcal{L}\left(\mathcal{A}\right)=\hat{\mathcal{A}}=\mathcal{A}\times_3 M_{I_3} \times_4 M_{I_4} \times_5 \dots \times_N M_{I_N},$$
				and its inverse 
				$$\mathcal{A}=\mathcal{L}^{-1}\left(\hat{\mathcal{A}}\right)=\hat{\mathcal{A}}\times_N M_{I_N}^{-1} \times_{N-1} M_{I_{N-1}}^{-1} \times_{N-2} \dots \times_3 M_{I_3}^{-1},$$
				where $M_{I_i}=W_{I_i}^{-1}C_{I_i}\left(I_{I_i}+Z_{I_i}\right)$ with $W_{I_i}={\tt diag}(C_{I_i}(:,1))$ and  $C_{I_i}$ is the matrix of discrete cosine and $Z_{I_i}$, $i =3, 4, \dots , N$ was  already defined.
			\end{definition}
		\medskip
		
		\noindent 
			In the last theorem we proved that a block-Toeplitz-plus-Hankel matrix of  a scalar-tensor is diagonalizable and in the next theorem we will prove that a  block-Toeplitz-plus-Hankel matrix of an $N^{th}$-order tensor is block diagonalizable.
			\medskip
			\begin{theorem}
				Let $\mathcal{X}\in \mathbb{K}^{I_1 \times I_2}_{I_3 \times \dots \times I_N}$, then  its block Toeplitz-plus-Hankel matrix is block diagonalizable and
				\begin{equation}\label{eq 20}
				{\tt btph}(\mathcal{X})=\left(C_{I_N}^*\otimes \dots \otimes C_{I_3}^* \otimes I_{I_1}\right) {\tt bdiag}\left(\mathcal{L}(\mathcal{X})\right) \left(C_{I_N}\otimes \dots \otimes C_{I_3} \otimes I_{I_2}\right)
				\end{equation}
			\end{theorem}	
			\begin{proof}
		For a tensor $\mathcal{X}\in \mathbb{K}^{I_1 \times I_2}_{I_3 \times \dots \times I_K }$ we have
				\begin{eqnarray*}
					{\tt btph}(\mathcal{X})=\sum_{i_1=1}^{I_1}\sum_{i_2=1}^{I_2}{\tt btph}(\mathcal{X}(i_1,\, i_2,\, :,\dots,\, :)) \otimes e_{i_1}e_{i_2}^T,
				\end{eqnarray*}
				where  $e_{i}$ is the $i$-th canonical vector of size $I_i$ for $i=1, 2$. Therefore, using Thorem \ref{theo 8}, we get  
				\begin{eqnarray*}
					{\tt btph}(\mathcal{X})&=&\sum_{i_1, i_2=1}^{I_1; I_2}\left(C_{I_N}^*\otimes \dots \otimes C_{I_3}^* \right) {\tt bdiag}\left(\mathcal{L}(\mathcal{X}(i_1,\, i_2,\, :,\, ,\dots,\, :))\right)\left(C_{I_N}\otimes  \dots \otimes C_{I_3} \right) \otimes e_{i_1}e_{i_2}^T\\
					&=&  \left(C_{I_N}^*\otimes  \dots \otimes C_{I_3}^* \otimes I_{I_1}\right) {\tt bdiag}(\mathcal{L}\left(\mathcal{X}\right)) \left(C_{I_N}\otimes \dots \otimes C_{I_3} \otimes I_{I_2}\right).
				\end{eqnarray*}
			\end{proof}
			\noindent Next, we define the generalized c-product.
			\begin{definition}
				\label{def 9}
				Let $\mathcal{A}\in \mathbb{K}^{I_1 \times l}_{I_3 \times \dots \times I_N}$ and $\mathcal{B}\in \mathbb{K}^{l\times I_2}_{I_3 \times \dots \times I_N}$. The generalized c-product $\mathcal{A}*_c\mathcal{B}$ is the tensor of $\mathbb{K}^{I_1 \times I_2}_{I_3 \times \dots \times I_N}$ defined as follows
				\begin{equation}
				\mathcal{A}*_c\mathcal{B}=ten\left({\tt btph}(\mathcal{A}){\tt btph}(\mathcal{B})\right).
				\end{equation}
			\end{definition}
		In the next definition, we generalize the face-wice product that was already defined for the third-order tensors.
			\begin{definition}{Generalized face-wise product}\\
				Let $\mathcal{A}\in \mathbb{K}^{I_1 \times l}_{I_3 \times \dots \times I_N}$ and $\mathcal{B}\in \mathbb{K}^{l\times I_2}_{I_3 \times \dots \times I_N}$, we define the face-wise product between $\mathcal{A}$ and $\mathcal{B}$ by the tensor $\mathcal{C}=\mathcal{A}\triangle \mathcal{B} \in \mathbb{K}^{I_1 \times I_2}_{I_3 \times \dots \times I_N}$, where the $p^{th}$ representative matrix of $\mathcal{C}$, is computed by the product of the $p^{th}$ representative matrices of $\mathcal{A}$ and $\mathcal{B}$, respectively. i.e.,
				\begin{equation}
				C^p=A^p B^p, \; \; p=1,\, 2,\, \dots ,\, P,
				\end{equation}
			where  the matrices $A^p$, $B^p$ and $C^p$ for $p=1, 2, \dots, P$ are the representative matrices given by \eqref{eq 3.5} of $\mathcal{A}$, $\mathcal{B}$ and $\mathcal{C}$, respectively.\\
			\end{definition}
		\medskip
			\begin{lemm}
				The generalized c-product of  two $N^{th}$-order tensors $\mathcal{A}\in \mathbb{K}^{I_1 \times l}_{I_3 \times \dots \times I_N}$ and $\mathcal{B}\in \mathbb{K}^{l\times I_2}_{I_3 \times \dots \times I_N}$ can be also computed in the cosine domain by 
				\begin{equation}
				\mathcal{A}*_c\mathcal{B}=\mathcal{L}^{-1}\left(\mathcal{L}(\mathcal{A})\triangle \mathcal{L}(\mathcal{B})\right).
				\end{equation}
			\end{lemm}
		
			\begin{proof}
				The generalized c-product of the two  tensors $\mathcal{A}\in \mathbb{K}^{I_1 \times l}_{I_3 \times \dots I_N}$ and $\mathcal{B}\in \mathbb{K}^{l \times I_2}_{I_3 \times \dots I_N}$,  is given by 
				$$\mathcal{A}*_c\mathcal{B}=ten\left({\tt btph}(\mathcal{A}){\tt btph}(\mathcal{B})\right)$$
				Therefore, using the notation  ${\textbf{C}}_N^{\,p}=\left(C_{I_N}\otimes  \dots \otimes C_{I_3} \otimes I_{p}\right)$ for  $p\geq 1$ and $N\geq 3$, we get   
				\begin{eqnarray*}
					{\tt btph}\left(\mathcal{A}*_c \mathcal{B}\right)&=&{\tt btph}\left(\mathcal{A}\right){\tt btph}\left(\mathcal{B}\right)\\
					&=& \left(\textbf{C}_N^{\,I_1}\right)^* \textbf{C}_N^{\,I_1} {\tt btph}\left(\mathcal{A}\right) \left(\textbf{C}_N^{\,l}\right)^* \textbf{C}_N^{\,l} {\tt btph}\left(\mathcal{B}\right) \left(\textbf{C}_N^{I_2}\right)^* \textbf{C}_N^{\,I_2}\\
					&=& \left(\textbf{C}_N^{\,I_1}\right)^* {\tt bdiag}(\mathcal{L}\left(\mathcal{A}\right)) {\tt bdiag}(\mathcal{L}\left(\mathcal{B}\right))\textbf{C}_N^{\,I_2}\\
					&=& {\tt btph}\left(\mathcal{L}^{-1}\left(\mathcal{L}\left(\mathcal{A}\right)\triangle \mathcal{L}\left(\mathcal{A}\right)\right)\right),  
				\end{eqnarray*}
				It follows that 
				\begin{equation*}
				\mathcal{A}*_c\mathcal{B}=\mathcal{L}^{-1}\left(\mathcal{L}(\mathcal{A})\triangle \mathcal{L}(\mathcal{B})\right).
				\end{equation*}
			\end{proof}
		\end{subsection}
	
		\begin{subsection}{Generalized tensor-tensor product} 
In this subsection, we define a general tensor-tensor $\mathcal{L}$ product for high-order tensor. We first define $\mathcal{L}$ as the following operator 
			\begin{eqnarray*}
		\mathcal{L}: \mathbb{K}^{I_1 \times I_2}_{I_3 \times \dots \times I_N} & \longrightarrow &  \mathbb{K}^{I_1 \times I_2}_{I_3 \times \dots \times I_N}\\
		\mathcal{A} & \longmapsto &  \mathcal{L}\left(\mathcal{A}\right)=\mathcal{A}\times_3 M_3 \dots \times_N M_N 
		\end{eqnarray*} 
		with $M_i \in \mathbb{K}^{I_i \times I_i}$ such that  $M_i=\alpha_i R_i$ for $i=3,\, 4,\, \dots ,\, N$, where $\alpha_i >0$ and $R_i$ is an unitary matrix. The inverse operator of $\mathcal{L}$ is defined as $\mathcal{L}^{-1}\left(\mathcal{A}\right)=\mathcal{A}\times_3 M_3^{-1} \times_4 \dots \times_N M_N^{-1}$. We will denote $\alpha=\alpha_3 \alpha_4\dots \alpha_N.$ 
		Next,  we   need the relation between the norm of a tensor  and its norm in the transformed domain (for example Fourier or cosine) given by 
			\begin{equation}
			\left\Vert \mathcal{A}\right\Vert_F=\dfrac{1}{\sqrt{\alpha}}\left\Vert \mathcal{L}\left(\mathcal{A}\right)\right\Vert_F,
			\end{equation}
		and we also have 
			\begin{equation}
			\left<\mathcal{A},\, \mathcal{B}\right>=\dfrac{1}{\alpha} \left<\mathcal{L}\left(\mathcal{A}\right),\, \mathcal{L}\left(\mathcal{B}\right) \right>, 
			\end{equation}
	In the next and for any $N^{th}$-order tensor $\mathcal{X}$,  we will denote the $p^{th}$ representative matrix \eqref{eq 3.5} of $\mathcal{L}\left(\mathcal{X}\right)$ by $\mathcal{L}(\mathcal{X})^{p}$. Now we can  define the generalized $*_\mathcal{L}$-product
\medskip	
			\begin{definition}
				Let $\mathcal{L}$ be the operator defined  above, then  the generalized $*_\mathcal{L}$-product of two $N^{th}$-order tensors $\mathcal{A} \in \mathbb{K}^{I_1 \times l}_{I_3 \times \dots I_N}$ and $\mathcal{B} \in \mathbb{K}^{l \times I_2}_{I_3 \times \dots \times I_N}$ is given by 
				\begin{equation}
				\mathcal{A}*_{\mathcal{L}}\mathcal{B}=\mathcal{L}^{-1}\left(\mathcal{L}\left(\mathcal{A}\right)\triangle \mathcal{L}\left(\mathcal{B}\right)\right) \in \mathbb{K}^{I_1 \times I_2}_{I_3 \times \dots \times I_N}
				\end{equation}
			\end{definition}
			The whole steps are   summarized in the following algorithm.
			\begin{algorithm}[H]
				\caption{The $*_\mathcal{L}$-product.}
				\label{alg:1}
				\begin{algorithmic}[1]
					\STATE \textbf{Inputs:} $\mathcal{A}\in \mathbb{K}^{n_1\times n}_{  n_3 \times \dots \times n_N}$ and $\mathcal{B}\in \mathbb{K}^{n\times n_2}_{ n_3\times \dots \times n_N}$. 
					\STATE \textbf{Output:} $\mathcal{C}=\mathcal{A}*\mathcal{B}\in \mathbb{K}^{n_1\times n_2}_{n_3 \dots \times n_K}$ .
					\STATE Compute $\hat{\mathcal{A}}=\mathcal{L}\left(\mathcal{A}\right)$ and $\hat{\mathcal{B}}=\mathcal{L}\left(\mathcal{B}\right)$.
					\FOR {$i=1, \ldots ,P$}
					\STATE  $\hat{\mathcal{C}}^{i}=\hat{\mathcal{A}}^{i}\hat{\mathcal{B}}^{i}$
					\ENDFOR
					\STATE $\mathcal{C}=\mathcal{L}^{-1}\left(\hat{\mathcal{C}}\right).$
				\end{algorithmic}
			\end{algorithm}
			\begin{prop}
				Let $\mathcal{X}\in \mathbb{K}^{I_1 \times l}_{I_3\times \dots \times I_K}$ and $\mathcal{Y}\in \mathbb{K}^{l \times I_2}_{I_3\times \dots \times I_K}$ two $K^{th}$-order tensors. Then we have
				\begin{equation}
				\left(\mathcal{X}*_\mathcal{L}\mathcal{Y}\right)_{i,j}=\sum_{k=1}^{l}\mathcal{X}_{i,k}*_\mathcal{L}\mathcal{Y}_{k,j},\; \text{for}\; 1\leq i\leq I_1 \; \text{and}\; 1\leq j\leq I_2,
				\end{equation}
				where $\mathcal{X}_{i,j}=\mathcal{X}(i,j,:,:,\dots,:)\in \mathbb{K}_{I_3\times I_4 \times \dots \times I_K}$  is the $(i,j)^{th}$ scalar-tensor of $\mathcal{X}$.
			\end{prop}
			\begin{proof}
			For or $1\leq i\leq I_1$ and $1\leq j\leq I_2$, we have
				\begin{eqnarray*}
					\left(\mathcal{X}*_\mathcal{L}\mathcal{Y}\right)_{i,j}&=&\mathcal{L}^{-1}\left(\mathcal{L}\left(\mathcal{X}\right)\triangle \mathcal{L}\left(\mathcal{Y}\right)\right)_{i,j}\\
					&=&\mathcal{L}^{-1}\left(\sum_{k=1}^{l}\mathcal{L}\left(\mathcal{X}\right)_{i,k}\triangle \mathcal{L}\left(\mathcal{Y}\right)_{k,j}\right)\\
					&=&\sum_{k=1}^{l}\mathcal{L}^{-1}\left(\mathcal{L}\left(\mathcal{X}_{i,k}\right)\triangle \mathcal{L}\left(\mathcal{Y}_{k,j}\right)\right)\\
					&=& \sum_{k=1}^{l} \mathcal{X}_{i,k}*_\mathcal{L}\mathcal{Y}_{k,j},
				\end{eqnarray*}
			which shows the result.
			\end{proof}
			\begin{prop}
			Let  $\mathcal{A}\in \mathbb{K}^{I_1 \times l}_{I_3\times \dots \times I_K}$ and $\mathcal{B}\in \mathbb{K}^{l\times I_2}_{I_3\times \dots \times I_K}$ be two $K^{th}$-order tensors, then we can express the $*_\mathcal{L}$-product of $\mathcal{A}$ and $\mathcal{B}$ as
				\begin{equation}
				\mathcal{A}*_\mathcal{L}\mathcal{B}=\left[\mathcal{A}*_\mathcal{L} \overrightarrow{\mathcal{B}}_1,\, \mathcal{A}*_\mathcal{L} \overrightarrow{\mathcal{B}}_2,\, \dots ,\,\mathcal{A}*_\mathcal{L} \overrightarrow{\mathcal{B}}_{I_2}\right],
				\end{equation}
				where $\overrightarrow{\mathcal{B}}_{i_2}=\mathcal{B}(:,i_2,:,\dots,:)\in \mathbb{K}^l_{I_3\times I_4 \times \dots \times I_K};\; i_2=1,\ldots,I_2$.
			\end{prop}
		\medskip
			\begin{proof}
			From the definition of the  $*_\mathcal{L}$-product of $\mathcal{A}$ and $\mathcal{B}$ and for $1\leq i_2 \leq I_2$, we get 
				\begin{eqnarray}
				\overrightarrow{\left(\mathcal{A}*_\mathcal{L}\mathcal{B}\right)}_{i_2}&=&\overrightarrow{\mathcal{L}^{-1}\left(\mathcal{L}\left(\mathcal{A}\right)\triangle \mathcal{L}\left(\mathcal{B}\right)\right)}_{i_2}\\
				&=&\mathcal{L}^{-1}\left(\mathcal{L}\left(\mathcal{A}\right)\triangle \overrightarrow{\mathcal{L}\left(\mathcal{B}\right)}_{i_2}\right)\\
				&=&\mathcal{L}^{-1}\left(\mathcal{L}\left(\mathcal{A}\right)\triangle \mathcal{L}\left(\overrightarrow{\mathcal{B}}_{i_2}\right)\right)\\
				&=& \mathcal{A}*_\mathcal{L}\overrightarrow{\mathcal{B}}_{i_2},
				\end{eqnarray}
			which gives the desired result.
			\end{proof}
		
		\medskip
	\noindent 	Related to the generalized $*_\mathcal{L}$-product , we give the definitions of the identity, transpose and orthogonal tensors.
			\begin{definition}{(The identity tensor)}\\
			The tensor identity tensor  $\mathcal{I}\in \mathbb{K}^{I_1 \times I_1}_{I_3 \times \dots I_N}$ is  such that $\left( \mathcal{L}(\mathcal{I})\right)^{p}=I$ for $p=1, 2, \dots, P$ where $\left(  \mathcal{L}(\mathcal{I}) \right)^{p}$ is the $p^{th}$ representative matrix of $\mathcal{L}\left(\mathcal{I}\right)$.\\
			\end{definition}
			From the previous definition, we can conclude that if  $\mathcal{A}\in \mathbb{K}^{I_1\times I_2}_{I_3\times \dots \times I_N}$, then  $\mathcal{A}*_{\mathcal{L}}\mathcal{I}=\mathcal{I}*_{\mathcal{L}}\mathcal{A}=\mathcal{A}$; because $\mathcal{A}*_{\mathcal{L}}\mathcal{I}=\mathcal{L}^{-1}\left(\mathcal{L}(\mathcal{A})\triangle \mathcal{L}(\mathcal{I})\right)=\mathcal{L}^{-1}\left(\mathcal{L}(\mathcal{A})\right),$ and the same for $\mathcal{I}*_{\mathcal{L}}\mathcal{A}.$
			\begin{definition}{(The transpose)}\\
				Let $\mathcal{A}\in \mathbb{K}^{I_1 \times I_2}_{I_3 \times \dots \times I_N}$,  the transpose  $\mathcal{A}^T$ of the tensor $\mathcal{A}$ is such that  $\mathcal{L}\left(\mathcal{A}^T\right)^p=\left(\mathcal{L}\left(\mathcal{A}\right)^p\right)^T$ for $p=1, 2, \dots , P$.\\
			\end{definition}
			This definition ensure the multiplication reversal property for the transpose under the $*_\mathcal{L}$-product, i.e., for $\mathcal{A}$ and $\mathcal{B}$ two $N^{th}$-order tensors of appropriate sizes, we get $\left(\mathcal{A}*_{\mathcal{L}}\mathcal{B}\right)^T=\mathcal{B}^T*_{\mathcal{L}}\mathcal{A}^T$; for explanation we have
			\begin{eqnarray*}
				\mathcal{L}\left(\left(\mathcal{A}*_\mathcal{L} \mathcal{B}\right)^T\right)^p=\left(\mathcal{L}\left(\mathcal{A}*_\mathcal{L}\mathcal{B}\right)^p\right)^T&=&\left(\mathcal{L}\left(\mathcal{A}\right)^p\mathcal{L}\left(\mathcal{B}\right)^p\right)^T \\
				&=&\left(\mathcal{L}\left(\mathcal{B}\right)^p\right)^T\left(\mathcal{L}\left(\mathcal{A}\right)^p\right)^T\\
				&=&\left(\mathcal{L}\left(\mathcal{B}^T\right)\right)^p\left(\mathcal{L}\left(\mathcal{A}^T\right)\right)^p\\
				&=& \mathcal{L}\left(\mathcal{B}^T*_\mathcal{L}\mathcal{A}^T\right)^p
			\end{eqnarray*}
			\begin{definition}{(Orthogonal tensor)}\\
				The tensor  $\mathcal{Q}\in \mathbb{K}^{I_1 \times I_1}_{I_3 \times \dots \times I_N}$ is  orthogonal  under the $*_\mathcal{L}$-product, iff  $\mathcal{Q}*_{\mathcal{L}}\mathcal{Q}^T=\mathcal{Q}^T *_{\mathcal{L}}\mathcal{Q}=\mathcal{I}$ which means that for each $i\in \{1,\, 2,\, \dots ,\, P\}$, $\mathcal{L}\left(\mathcal{Q}\right)^i$ is an orthogonal matrix.\\
			\end{definition}
		
		\noindent Notice that if $\mathcal{Q}\in \mathbb{K}^{I_1 \times I_1}_{I_3 \times \dots \times I_N}$ is  orthogonal, then   for an  $N^{th}$-order tensor $\mathcal{A}$ of an appropriate size, we have  
			\begin{eqnarray}
			\left\Vert \mathcal{A}*_\mathcal{L}\mathcal{Q}\right\Vert_F^2=\dfrac{1}{c}\left\Vert \mathcal{L}\left(\mathcal{A}\right)\triangle \mathcal{L}\left(\mathcal{Q}\right)\right\Vert_F^2&=&\dfrac{1}{c}\sum_{i=1}^{P}\left\Vert \mathcal{L}\left(\mathcal{A}\right)^i  \mathcal{L}\left(\mathcal{Q}\right)^i\right\Vert_F^2 \\
			&=&\dfrac{1}{c}\sum_{i=1}^{P}\left\Vert \mathcal{L}\left(\mathcal{A}\right)^i  \right\Vert_F^2 \\
			&=& \left\Vert \mathcal{A}\right\Vert_F^2
			\end{eqnarray}
			\begin{definition}{f-diagonal tensor}\\
			An $N^{th}$-order tensor $\mathcal{X}$ is f-diagonal, if each $\mathcal{L}(\mathcal{X})^p$ is a diagonal matrix for all $p$ in $\{1, 2, \dots, P \}$.
			\end{definition}
		
		\noindent 
			\begin{theorem}
			The set 	$\left( \mathbb{K}_{I_3\times I_4 \times \dots \times I_N},+,*_\mathcal{L} \right)$ is a commutaive ring.
			\end{theorem}
			\begin{proof}
				It is easy to prove that $\left(\mathbb{K}_{I_3 \times I_4 \times \dots \times I_N},+\right)$ is an abelian group with $\textbf{0}\in \mathbb{K}_{I_3 \times I_4 \times \dots \times I_N}$ as a neutral element. On the other hand, we have 
				\begin{itemize}
					\item Let $\mathcal{A}, \, \mathcal{B}$ and $\mathcal{C}$ are scalar-tensors in $\mathbb{K}_{I_3 \times I_4 \times \dots \times I_N}$, then we will get 
					\begin{eqnarray*}
						\mathcal{A}*_\mathcal{L}\left(\mathcal{B}*_\mathcal{L}\mathcal{C}\right)&=&\mathcal{L}^{-1}\left(\mathcal{L}\left(\mathcal{A}\right)\triangle \mathcal{L}\left(\mathcal{L}^{-1}\left(\mathcal{L}\left(\mathcal{B}\right)\triangle \mathcal{L}\left(\mathcal{C}\right)\right)\right)\right)\\
						&=& \mathcal{L}^{-1}\left(\mathcal{L}\left(\mathcal{A}\right)\triangle \mathcal{L}\left(\mathcal{B}\right)\triangle \mathcal{L}\left(\mathcal{C}\right)\right)\\
						&=& \mathcal{L}^{-1}\left(\mathcal{L}\left(\mathcal{L}^{-1}\left(\mathcal{L}\left(\mathcal{A}\right)\triangle \mathcal{L}\left(\mathcal{B}\right)\right)\right)\triangle \mathcal{L}\left(\textbf{c}\right)\right)\\
						&=&\left(\mathcal{A} *_\mathcal{L}\mathcal{B}\right)*_\mathcal{L}\mathcal{C}.
					\end{eqnarray*}
					\item We also  have 
					\begin{eqnarray*}
						\mathcal{A}*_\mathcal{L}\left(\mathcal{B}+\mathcal{C}\right)&=&\mathcal{L}^{-1}\left(\mathcal{L}\left(\mathcal{A}\right)\triangle \mathcal{L}\left(\mathcal{B}+\mathcal{C}\right)\right)\\
						&=&\mathcal{L}^{-1}\left(\mathcal{L}\left(\mathcal{A}\right)\triangle\left( \mathcal{L}\left(\mathcal{B}\right)+\mathcal{L}\left(\mathcal{C}\right)\right)\right)\\
						&=&\mathcal{L}^{-1}\left(\mathcal{L}\left(\mathcal{A}\right)\triangle \mathcal{L}\left(\mathcal{B}\right)+\mathcal{L}\left(\mathcal{A}\right)\triangle\mathcal{L}\left(\mathcal{C}\right)\right)\\
						&=&\mathcal{L}^{-1}\left(\mathcal{L}\left(\mathcal{A}\right)\triangle \mathcal{L}\left(\mathcal{B}\right)\right)+\mathcal{L}^{-1}\left(\mathcal{L}\left(\mathcal{A}\right)\triangle\mathcal{L}\left(\mathcal{C}\right)\right)\\
						&=& \mathcal{A}*_\mathcal{L}\mathcal{B} +\mathcal{A}*_\mathcal{L}\mathcal{C}.
					\end{eqnarray*}
					\item Finally, the commutativity is also satisfied by 
					\begin{eqnarray*}
						\mathcal{A}*_\mathcal{L}\mathcal{B}=\mathcal{L}^{-1}\left(\mathcal{L}\left(\mathcal{A}\right)\triangle\mathcal{L}\left(\mathcal{B}\right)\right)&=&\mathcal{L}^{-1}\left(\mathcal{L}\left(\mathcal{B}\right)\triangle\mathcal{L}\left(\mathcal{A}\right)\right)\\
						&=& \mathcal{B}*_\mathcal{L}\mathcal{A}.
					\end{eqnarray*}
				\end{itemize}
			\end{proof}
			\begin{theorem}{(The tensor $*_\mathcal{L}$-SVD)}\\
				Let $\mathcal{A}\in \mathbb{K}^{I_1 \times I_2}_{I_3 \times \dots \times I_N}$ be an $N^{th}$-order tensor, then $\mathcal{A}$ can be decomposed  as 
				\begin{equation}\label{tsvd1}
				\mathcal{A}=\mathcal{U}*_{\mathcal{L}}\mathcal{S}*_{\mathcal{L}}\mathcal{V}^T=\sum_{k=1}^{r}\mathcal{U}(:,k,:,\dots,:)*_\mathcal{L}\mathcal{S}(k,k,:,\dots,:)*_\mathcal{L}\mathcal{V}(:,k,:,\dots,:)^T,
				\end{equation}
				where $\mathcal{U}\in \mathbb{K}^{I_1 \times I_1}_{I_3 \times \dots \times I_N}$ and $\mathcal{V}\in \mathbb{K}^{I_2 \times I_2}_{I_3 \times \dots \times I_N}$ are orthogonal tensors, $\mathcal{S}\in \mathbb{K}^{I_1 \times I_2}_{I_3 \times \dots \times I_N}$ is an f-diagonal and $r$ is the tubal rank of $\mathcal{A}$ which will be defined in the next.
			\end{theorem}
			\begin{proof}
				For a $K^{th}$-order tensor $\mathcal{A}$ and $p=1,\, 2,\, \dots ,\, P$, consider SVD decomposition of the matrices 
				\begin{equation*}
				\mathcal{L}\left(\mathcal{A}\right)^p=U_\mathcal{L}^p S_\mathcal{L}^p \left(V_\mathcal{L}^p\right)^T.
				\end{equation*} 
				Then we obtain 
				\begin{equation*}
				\mathcal{L}\left(\mathcal{A}\right)=\mathcal{U}_\mathcal{L} \triangle \mathcal{S}_\mathcal{L} \triangle \left(\mathcal{V}_\mathcal{L}\right)^T,
				\end{equation*}
				with  $\mathcal{L}\left(\mathcal{U}\right)=\mathcal{U}_\mathcal{L}$, $\mathcal{L}\left(\mathcal{S}\right)=\mathcal{S}_\mathcal{L}$ and $\mathcal{L}\left(\mathcal{V}\right)=\mathcal{V}_\mathcal{L}$, where those tensors are well defined since the operator $\mathcal{L}$ is inverstible. Therefore
				\begin{equation*}
				\mathcal{L}\left(\mathcal{A}\right)= \mathcal{L}\left(\mathcal{U}\right)  \triangle
				\mathcal{L}\left(\mathcal{S}\right)  \triangle \mathcal{L}\left(\mathcal{V}\right)^T \Longleftrightarrow  \mathcal{A}= \mathcal{L}^{-1}\left(\mathcal{L}\left(\mathcal{U}\right)  \triangle
				\mathcal{L}\left(\mathcal{S}\right)  \triangle \mathcal{L}\left(\mathcal{V}\right)^T\right)
				\end{equation*}
				which gives
				\begin{equation*}
				\mathcal{A}=\mathcal{U}*_\mathcal{L} \mathcal{S}*_\mathcal{L} \mathcal{V}^T.
				\end{equation*}
			\end{proof}
		
			The following algorithm summarises the different steps for computing the tensor $*_\mathcal{L}$-SVD of an $N^{th}$-order tensor.
			\begin{algorithm}[H]
				\caption{The $*_\mathcal{L}$-svd.}
				\label{alg:3}
				\begin{algorithmic}[1]
					\STATE \textbf{Inputs:} $\mathcal{A}\in \mathbb{K}^{I_1\times I_2}_{I_3\times \dots \times I_N}$. 
					\STATE \textbf{Output:} $\mathcal{U}\in \mathbb{K}^{I_1\times I_1}_{I_3\times \dots \times I_N}$,  $\mathcal{S}\in \mathbb{K}^{I_1\times I_2}_{I_3\times \dots \times I_N}$, $\mathcal{V}\in \mathbb{K}^{I_2\times I_2}_{I_3\times \dots \times I_N}$.
					\STATE \textbf{Compute} $\hat{\mathcal{A}}=\mathcal{L}\left(\mathcal{A}\right)$.
					\FOR {$i=1, \ldots ,P$}
					\STATE  $\left[\hat{\mathcal{U}}^{i}, \hat{\mathcal{S}}^{i}, \hat{\mathcal{V}}^{i} \right]=svd(\hat{\mathcal{A}}^{i})$
					\ENDFOR
					\STATE $\mathcal{U}=\mathcal{L}^{-1}\left(\hat{\mathcal{U}}\right)$, $\mathcal{S}=\mathcal{L}^{-1}\left(\hat{\mathcal{S}}\right)$ and $\mathcal{V}=\mathcal{L}^{-1}\left(\hat{\mathcal{V}}\right)$
				\end{algorithmic}
			\end{algorithm}
			\begin{coro}
				Let $\mathcal{A}$ be  an $N^{th}$-order tensor in $\mathbb{K}^{I_1\times I_2}_{I_3\times \dots \times I_N}$,  and let $\mathcal{S}_i=\mathcal{S}(i,i,:,\dots,:)$, where $\mathcal{S}$ is given by \eqref{tsvd1}. Then  
				\begin{equation}
				\left\Vert \mathcal{A}\right\Vert_F^2=\left\Vert \mathcal{S}\right\Vert_F^2=\sum_{i=1}^{\min(I_1,\, I_2)} \left\Vert \mathcal{S}_i\right\Vert_F^2.
				\end{equation}
				Furthermore
				\begin{equation}
				\left\Vert \mathcal{S}_1 \right\Vert_F\geq \left\Vert \mathcal{S}_2 \right\Vert_F \geq \dots \geq 0.
				\end{equation}
			\end{coro}
			\begin{proof}
				Since the tensors $\mathcal{U}$ and $\mathcal{V}$ given from the $*_\mathcal{L}$-svd are orthogonal, then  
				\begin{eqnarray*}
					\left\Vert \mathcal{A}\right\Vert_F^2 =\left\Vert \mathcal{U}*_\mathcal{L}  \mathcal{S}*_\mathcal{L}\mathcal{V}^T\right\Vert_F^2&=&\left\Vert \mathcal{S}\right\Vert_F^2 \\
					&=& \dfrac{1}{\alpha}\sum_{p=1}^{P}\left\Vert \mathcal{L}\left(\mathcal{S}\right)^p\right\Vert_F^2 \\
					&=& \dfrac{1}{\alpha}\sum_{p=1}^{P} \sum_{i=1}^{\min(I_1,\, I_2)}\left( \mathcal{L}\left(\mathcal{S}\right)^p(i,i)\right)^2 \\
					&=&  \sum_{i=1}^{\min(I_1,\, I_2)}\left\Vert \mathcal{S}_i \right\Vert_F^2.
				\end{eqnarray*}
			On the other hand,
				\begin{eqnarray*}
					\left\Vert \mathcal{S}_i \right\Vert_F^2=\dfrac{1}{\alpha}\left\Vert \mathcal{L}\left(\mathcal{S}\right)_i \right\Vert_F^2=\dfrac{1}{\alpha}\sum_{p=1}^{P} \left( \mathcal{L}\left(\mathcal{S}\right)^p(i,i)\right)^2 \geq \dfrac{1}{\alpha}\sum_{p=1}^{P}\left( \mathcal{L}\left(\mathcal{S}\right)^p(i+1,i+1)\right)^2=\left\Vert \mathcal{S}_{i+1}\right\Vert_F^2.
				\end{eqnarray*}
			\end{proof}
		\medskip
		
	\noindent 	Next, we give different definitions of a rank of a high-order tensor.
		\begin{definition}{(The $*_\mathcal{L}$-tubal rank)}\\
			Let $\mathcal{A}\in \mathbb{K}^{I_1 \times I_2}_{I_3 \times \dots \times I_N}$,  then  the tensor $*_\mathcal{L}$-tubal rank is defined as 
			\begin{equation}
				{\tt rank}_t(\mathcal{A})={\tt card}\{i/ \; \mathcal{S}(i,i,:,\dots,:)\neq \textbf{0} \},
			\end{equation}
			where $\mathcal{S}$ is the $f$-diagonal tensor given from the $*_\mathcal{L}$-svd of $\mathcal{A}$ \eqref{tsvd1}.
		\end{definition}
	\medskip
		\begin{lemm}
			For a $K^{th}$-order tensor,  the $*_\mathcal{L}$-tubal rank can be written as
			\begin{equation}
				{\tt rank}_t\left(\mathcal{A}\right)={\tt card}\left\{i/ \exists p \in \{1,\, 2,\, \dots ,\, P\}; \; \mathcal{L}\left(\mathcal{S}\right)^p(i,i)\neq 0 \right\}.
			\end{equation}
		\end{lemm}
		\begin{proof}
			The key idea for proving the above result is the following equivalence,  
			\begin{equation*}
				\mathcal{S}\left(i,i,:,\dots,:\right)=\textbf{0} \Longleftrightarrow \forall p \in \{1,2,\dots, P\}; \; \mathcal{L}\left(\mathcal{S}\right)^p(i,i)=0,\, i =1,2, \dots , \min(I_1,I_2),
			\end{equation*}
		which is equivalent to 
			\begin{equation*}
				\mathcal{S}\left(i,i,:,\dots,:\right)\neq \textbf{0} \Longleftrightarrow \exists p \in \{1,2,\dots, P\}; \; \mathcal{L}\left(\mathcal{S}\right)^p(i,i)\neq 0, \; ,\, i =1,2, \dots , \min(I_1,I_2).
			\end{equation*}
	Therefore, 
		\begin{eqnarray}
			\mathcal{S}\left(i,i,:,\dots,:\right)\neq \textbf{0} \Longleftrightarrow \underset{1\leq p\leq P}{\max}\, \mathcal{L}\left(\mathcal{S}\right)^p(i,i)\neq 0, \, i =1,2, \dots , \min(I_1,I_2).
		\end{eqnarray}
		\end{proof}
		\begin{definition}{(Multirank and average rank)}\\
			For an $N^{th}$-order tensor $\mathcal{A}$ of size $I_1 \times I_2 \times \dots \times I_N$, its multirank under the $*_\mathcal{L}$-product is defined as the vector $\rho$ of size $P$, where its $i^{th}$ element is the rank of $\mathcal{L}\left(\mathcal{A}\right)^i$, i.e.,
			\begin{equation}
				\rho_i= {\tt rank}(\mathcal{L}\left(\mathcal{A}\right)^i).
			\end{equation}
			The average rank of $\mathcal{A}$ is defined as the mean of the vector $\rho$, i.e., 
			\begin{equation}
				{\tt rank}_a\left(\mathcal{A}\right)= \dfrac{\displaystyle \sum_{i=1}^{P}\rho_i}{P}.
			\end{equation}
		\end{definition}
		\begin{remark}
		We notice that the average rank of an $N^{th}$-order tensor is defined as the rank of the block-diagonal matrix of $\mathcal{L}\left(\mathcal{A}\right)$ divided by $P$, 
			\begin{equation}
				{\tt rank}_a(\mathcal{A})=\dfrac{{\tt bdiag}(\mathcal{L}\left(\mathcal{A}\right))}{P}.
			\end{equation}
		\end{remark}

			Next, we give a generalized version of the well known Eckart Young using the $*_\mathcal{L}$-product .
			\begin{theorem}{(Eckart Young)}\\
				Let $\mathcal{A}\in \mathbb{K}^{I_1 \times I_2}_{I_3 \times \dots \times I_N}$ and $\mathcal{A}_k=\mathcal{U}(:,1:k,:,\dots ,:)*_\mathcal{L}\mathcal{S}(1:k,1:k,:,\dots,:)*_\mathcal{L}\mathcal{V}(1:k,:,\dots ,:)^T$, where $k$ is a nonzero  positive integer. Let $\mathcal{C}$  be the set befined by  $$\mathcal{C}=\left\{\mathcal{X}*_\mathcal{L}\mathcal{Y} \; / \mathcal{X} \in \mathbb{K}^{I_1 \times k}_{I_3 \times \dots \times I_N}, \; \mathcal{Y} \in \mathbb{K}^{k \times I_2}_{I_3 \times \dots \times I_N} \right\}.$$ 
				 Then the tensor $\mathcal{A}_k$ solves the minimisation problem
				$$\displaystyle \min_{\mathcal{X} \in \mathtt{C}}  \Vert \mathcal{A} -\mathcal{X} \Vert_F,$$
			and the error-norm is given by $$\left\Vert \mathcal{A}-\mathcal{A}_k \right\Vert_F^2=\sum_{i=k+1}^{r}\left\Vert \mathcal{S}_i \right\Vert_F^2.$$
			\end{theorem}
			\begin{proof}
				Let $\mathcal{B}=\mathcal{X}*_\mathcal{L}\mathcal{Y}\in \mathcal{C}$, then  we have
				$$\left\Vert \mathcal{A}-\mathcal{B}\right\Vert_F^2=\dfrac{1}{\alpha}\sum_{i=1}^{P} \left\Vert \mathcal{L}\left(\mathcal{A}\right)^i -\mathcal{L}\left(\mathcal{B}\right)^i \right\Vert_F^2.$$
				Now, since $\mathcal{L}\left(\mathcal{B}\right)^i=\mathcal{L}\left(\mathcal{X}\right)^i\mathcal{L}\left(\mathcal{Y}\right)^i$ and by using the  matrix  Eckart Young theorem  we obtain the result showing that  the $k$-best approximation  of the matrix $\mathcal{L}\left(\mathcal{B}\right)^i$ is given by  $\mathcal{L}\left(\mathcal{U}\right)^i(:,1:k)\mathcal{L}\left(\mathcal{S}\right)^i(1:k, 1:k)\mathcal{L}\left(\mathcal{V}\right)^i(1:k,:)^T$, where $\mathcal{U}$, $\mathcal{S}$ and $\mathcal{V}$ are given by \eqref{tsvd1}.
			\end{proof}
			\medskip
			\begin{prop}
				The tensor spectral norm of an $N^{th}$-order tensor $\mathcal{A}$ satisfies the following equation
				\begin{equation}
				\left\Vert \mathcal{A}\right\Vert=\left\Vert {\tt bdiag}\left(\mathcal{L}\left(\mathcal{A}\right)\right)\right\Vert
				\end{equation}
			\end{prop}
			\begin{proof}
				The spectral norm of a tensor $\mathcal{A}$, is defined by 
				\begin{eqnarray}
				\left\Vert \mathcal{A} \right\Vert=\underset{\underset{\left\Vert \mathcal{V}\right\Vert_F=1}{\mathcal{V}\in \mathbb{K}^{I_2 \times 1}_{I_3 \times \dots \times I_N}}}{sup} \left\Vert \mathcal{A}*_\mathcal{L}\mathcal{V}\right\Vert_F&=&\dfrac{1}{\sqrt{\alpha}}\underset{\underset{\left\Vert \mathcal{V}\right\Vert_F=1}{\mathcal{V}\in \mathbb{K}^{I_2 \times 1}_{I_3 \times \dots \times I_N}}}{sup} \left\Vert \mathcal{L}(\mathcal{A})\triangle \mathcal{L}(\mathcal{V})\right\Vert_F \\
				&=& \dfrac{1}{\sqrt{\alpha}}\underset{\underset{\left\Vert \mathcal{V}\right\Vert_F=1}{\mathcal{V}\in \mathbb{K}^{I_2 \times 1}_{I_3 \times \dots \times I_N}}}{sup} \left\Vert {\tt bdiag}\left(\mathcal{L}(\mathcal{A})\right) {\tt bdiag}\left(\mathcal{L}(\mathcal{V})\right)\right\Vert_F \\
				&=& \left\Vert {\tt bdiag}\left(\mathcal{L}(\mathcal{A})\right)\right\Vert
				\end{eqnarray}
			\end{proof}
			\begin{definition}{Tensor nuclear norm}\\
				The tensor nuclear norm of a tensor $\mathcal{A}\in \mathbb{K}^{I_1 \times I_2}_{I_3 \times \dots \times I_N}$ is defined as its dual norm, i.e.,
				\begin{equation}\label{eq 23}
				\left\Vert \mathcal{A}\right\Vert_*=\underset{\left\Vert\mathcal{B}\right\Vert\leq 1}{sup}\, \left\vert\left<\mathcal{A},\mathcal{B}\right>\right\vert
				\end{equation}
			\end{definition}
			\begin{theorem}\label{theo 21}
				The nuclear norm of a tensor $\mathcal{A}\in  \mathbb{K}^{I_1 \times I_2}_{I_3 \times \dots \times I_N}$ verifies the relation 
				\begin{equation} \label{eq 24}
				\left\Vert \mathcal{A}\right\Vert_*=\dfrac{1}{\alpha}\left\Vert {\tt bdiag}(\mathcal{L}\left(\mathcal{A}\right))\right\Vert_*=\dfrac{1}{\alpha}\left\Vert {\tt bdiag}(\mathcal{L}\left(\mathcal{S}\right))\right\Vert_*.
				\end{equation}
			\end{theorem}
			\begin{proof}
				Starting by the definition of the nuclear norm in \eqref{eq 23}, we will get 
				\begin{eqnarray*}
					\left\Vert \mathcal{A}\right\Vert_*=\underset{\left\Vert \mathcal{B}\right\Vert\leq 1}{\sup}\, \left\vert \left<\mathcal{A}, \mathcal{B}\right> \right\vert &=&\dfrac{1}{\alpha} \underset{\left\Vert \mathcal{B}\right\Vert\leq 1}{\sup}\, \left\vert \left<{\tt bdiag}\left(\mathcal{L}\left(\mathcal{A}\right)\right), {\tt bdiag}\left(\mathcal{L}\left(\mathcal{B}\right)\right)\right> \right\vert \\
					&=& \dfrac{1}{\alpha}\left\Vert {\tt bdiag}(\mathcal{L}\left(\mathcal{A}\right)) \right\Vert_*.  
				\end{eqnarray*}
				Which gives the first equality of \eqref{eq 24}, and the last equality is trivial since ${\tt bdiag}\left(\mathcal{L}\left(\mathcal{A}\right)\right)={\tt bdiag}\left(\mathcal{L}\left(\mathcal{U}\right)\right){\tt bdiag}\left(\mathcal{L}\left(\mathcal{S}\right)\right){\tt bdiag}\left(\mathcal{L}\left(\mathcal{V}^T\right)\right)$ where $\mathcal{U}$, $\mathcal{S}$ and $\mathcal{V}$ are given from the $*_\mathcal{L}$-svd of $\mathcal{A}$.
			\end{proof}
		\medskip
		
\noindent 	We can also express the tensor nuclear norm of a $K^{th}$-order tensor as 
			\begin{equation}
			\left\Vert \mathcal{A}\right\Vert_*=\dfrac{1}{\alpha}\sum_{p=1}^{P}\left(\sum_{i=1}^{\min(I_1,I_2)}\mathcal{L}\left(\mathcal{S}\right)^p(i,i)\right).
			\end{equation} 
			\begin{theorem}
				\label{theo21}
				 The envelope convex of the function tensor average rank  on  the set $\mathbb{S}=\left\{\mathcal{A}\in  \mathbb{K}^{I_1 \times I_2}_{I_3 \times \dots \times I_N} /\, \left\Vert \mathcal{A}\right\Vert \leq 1 \right\}$ is the tensor nuclear norm. 
			\end{theorem}
			\begin{proof}
				Let $\mathcal{A}\in \mathbb{S}$, then 
				$\left\Vert \mathcal{A}\right\Vert\leq 1$, 
				which means that for $i=1,\,2 ,\, \dots, \, P$, we have 
				$$\left\Vert \mathcal{L}\left(\mathcal{A}\right)^i\right\Vert \leq 1$$
				and by using the fact that the tensor average rank is the average of the tensor multirank, it follows (see  \cite{fazel2001rank})  that for each $i=1,\, 2,\, \dots ,\, P$, the envelope convex of $\rho_i$ is $\left\Vert \mathcal{L}\left(\mathcal{A}\right)^i \right\Vert_*$. \\
				Consequently, the envelope convex of the function average rank is the tensor nuclear norm.
			\end{proof}

	\medskip
\noindent		
Next,  we define the tensor singular value thresholding under the $*_\mathcal{L}$-product ($*_\mathcal{L}$-svt).
			\begin{definition}
				Let $\mathcal{A}\in \mathbb{K}^{I_1 \times I_2}_{I_3 \times \dots \times I_N}$ an $N^{th}$-order tensor and $\tau > 0$, then we call its tensor singular value thresholding the  following tensor 
				\begin{equation}
				\mathcal{D}_\tau \left(\mathcal{A}\right)=\mathcal{U}* \mathcal{S}_\tau * \mathcal{V}^T
				\end{equation}
				where $\mathcal{U}$, $\mathcal{S}$ and $\mathcal{V}$ are given from the geberalize t-svd \eqref{tsvd1}, and $\mathcal{S}_\tau=\mathcal{L}^{-1}\left((\mathcal{L}(\mathcal{S})-\tau)_+\right)$.
			\end{definition}
			We also have  the following important result that links the tensor nuclear norm and the tensor singular value thresholding.
				\medskip
			\noindent
			\begin{theorem}
				\label{theo23}
				Let $\mathcal{A}\in \mathbb{K}^{I_1 \times I_2}_{I_3 \times \dots \times I_N}$ and $\tau > 0$, then we have  
				\begin{eqnarray}\label{eq 26}
				\mathcal{D}_\tau\left(\mathcal{A}\right)=\underset{\mathcal{Y}\in \mathbb{K}^{I_1 \times I_2}_{I_3 \times \dots \times I_N}}{\arg\, \min}\, \tau\left\Vert \mathcal{Y}\right\Vert_* +\dfrac{1}{2}\left\Vert \mathcal{Y}-\mathcal{A}\right\Vert_F^2
				\end{eqnarray}
			\end{theorem}
			\begin{proof}
				Solving the optimization problem  \eqref{eq 26} is equivalent to solve the following one 
				\begin{equation}\label{eq 277}
				\underset{\mathcal{Y}\in \mathbb{K}^{I_1 \times I_2}_{I_3 \times \dots \times I_N}}{\arg \min}\, \dfrac{1}{\alpha}\left\{ \tau \left\Vert \ 
				{\tt bdiag}(\mathcal{\mathcal{L}}\left({\mathcal{Y}}\right))\right\Vert_* + \dfrac{1}{2}\left\Vert {\tt bdiag}(\mathcal{\mathcal{L}}\left({\mathcal{Y}}\right))-{\tt bdiag}(\mathcal{\mathcal{L}}\left({\mathcal{A}}\right))\right\Vert_F^2 \right\}.
				\end{equation}
				The optimization problem given in \eqref{eq 277} can be solved by solving $P$ subproblems independently, i.e., for each $i\in \{1,\, 2,\, \dots ,\, P\}$ we will try to find the solution of 
				\begin{equation}\label{eq 288}
				\underset{  \mathcal{\mathcal{L}}\left({\mathcal{Y}}\right)^i  \in \mathbb{K}^{I_1 \times I_2} }{\arg \min}\, \left\{ \tau \left\Vert \mathcal{\mathcal{L}}\left({\mathcal{Y}}\right)^i \right\Vert_* + \dfrac{1}{2}\left\Vert \mathcal{\mathcal{L}}\left({\mathcal{Y}}\right)^i -\mathcal{\mathcal{L}}\left({\mathcal{A}}\right)^i \right\Vert_F^2  \right\}.
				\end{equation}
				Using \cite{cai2010singular},  the solution of  each subproblem in \eqref{eq 288} is $\mathcal{L}\left(\mathcal{D}_\tau \left(\mathcal{A}\right)\right)^i$. Therefore, $\mathcal{D}_\tau \left(\mathcal{A}\right)$ solves \eqref{eq 26}.
			\end{proof}
		
			\medskip
		\noindent
			In the following algorithm we  give the different  steps of computing the tensor $*_\mathcal{L}$-svt
			\begin{algorithm}[H]
				\caption{The $*_\mathcal{L}$-svt.}
				\label{alg:5}
				\begin{algorithmic}[1]
					\STATE \textbf{Inputs:} $\mathcal{A}\in \mathbb{K}^{I_1\times I_2}_{I_3\times \dots \times I_N}$ and $\tau>0$. 
					\STATE \textbf{Output:} $\mathcal{D}_\tau\left(\mathcal{A}\right)\in \mathbb{K}^{I_1\times I_2}_{I_3\times \dots \times I_N}$.
					\STATE \textbf{Compute:} $\hat{\mathcal{A}}=\mathcal{L}\left(\mathcal{A}\right)$
					\FOR {$i=1, \ldots ,P$}
					\STATE  $\left[\hat{\mathcal{U}}^{i}, \hat{\mathcal{S}}^{i}, \hat{\mathcal{V}}^{i}\right]=svd(\hat{\mathcal{A}}^{i})$
					\STATE $\hat{\mathcal{S}}^i=\max(\hat{\mathcal{S}}^i-\tau,0)$.
					\STATE $\hat{\mathcal{A}}^i=\hat{\mathcal{U}}^i\hat{\mathcal{S}}^i(\hat{\mathcal{V}}^i)^T$
					\ENDFOR
					\STATE $\mathcal{A}=\mathcal{L}^{-1}(\hat{\mathcal{A}})$.
				\end{algorithmic}
			\end{algorithm}
		\end{subsection}
	\end{section}
	\begin{section}{Tensor completion using $*_\mathcal{L}$-product}
		\label{sec3}
		Tensor completion is the  problem that consists in finding some unknown pixels of the data from an observed data that contains  some known  pixels. Many algorithms have been developed the last years; see  \cite{bengua2017efficient,xu2013parallel,liu2012tensor}. Some of those methods  use regularization techniques such as the total variation regularization \cite{ji2016tensor,ding2019low}. Those algorithms suffer from the computationally costs and the slowness. In \cite{zhang2014novel,Ng2019fast} the problem of tensor completion using the  t-product and the c-product with  regularized  total variation gave  good results. The  problem of those methods is the fact that they are applied only to  third-order tensors. Next, we propose to extend those methods to high order tensors using the $*_\mathcal{L}$-product we defined in the preceeding section for tensors of order greater than three.\\
		The main optimization problem that solves the problem of tensor completion consists in finding a low-rank tensor that contains the main information (the known pixels), which depends on the  definition  of the rank that we will consider. In our proposed method, we will consider  the average rank of a tensor. Thus our main optimization problem is given as follows 
		\begin{eqnarray}\label{eq48}
		&\underset{\mathcal{X}\in \mathbb{K}^{I_1 \times I_2}_{I_3 \times \dots \times I_N}}{\arg \, \min}&\, {\tt rank}_a\left(\mathcal{X}\right) \nonumber \\
		&{\tt s.t}&  \mathcal{P}_\Omega\left(\mathcal{X}\right)=\mathcal{P}_\Omega\left(\mathcal{M}\right),		
		\end{eqnarray}
		where $\mathcal{X}$ is the underlying tensor, $\mathcal{M} \in \mathbb{K}^{I_1 \times I_2}_{I_3 \times \dots \times I_N}$ is the observed tensor, $\Omega$ is the set of the known pixels and $\mathcal{P}_\Omega$ is the projection operator that copy the values of the pixels onto $\Omega$. However, the optimization problem \eqref{eq48} is  NP-hard \cite{hillar2013most}. It is known that convex optimization problem are the easiest optimization problems to solve and for this reason we will use the approximation of the function average rank given  in Theorem \ref{theo21}. Our optimization problem is transformed to the following one
		\begin{eqnarray}\label{eq 49}
		&\underset{\mathcal{X}\in \mathbb{K}^{I_1 \times I_2}_{I_3 \times \dots \times I_N}}{\min}&\, \left\Vert \mathcal{X}\right\Vert_* \nonumber \\
		&{\tt s.t}& \mathcal{P}_\Omega\left(\mathcal{X}\right)=\mathcal{P}_\Omega\left(\mathcal{M}\right),
		\end{eqnarray}
		where $\left\Vert . \right\Vert_*$ is the tensor nuclear norm in Theorem \ref{theo 21}.  The main techniques for solving \eqref{eq 49} is the Proximal Gradient Algorithm (PGA). The problem  \eqref{eq 49}  will be solved iteratively as 
		\begin{equation}\label{eq 50}
		\mathcal{X}^{k+1}=\underset{\mathcal{X}\in \mathbb{K}^{I_1 \times I_2}_{I_3 \times \dots \times I_N}}{\arg\, \min}\mu \left\Vert \mathcal{X}\right\Vert_* + \dfrac{1}{2}\left\Vert \mathcal{X}-\mathcal{G}^k\right\Vert_F^2,
		\end{equation}
	 where 
		\begin{eqnarray}
		&\mathcal{G}^k&=\mathcal{Y}^k-\left(\mathcal{P}_\Omega(\mathcal{Y}^k)-\mathcal{P}_\Omega(\mathcal{M})\right)=\mathcal{P}_{\Omega^c}(\mathcal{Y}^k)-\mathcal{P}_\Omega(\mathcal{M}),\, and  \label{eq51} \\
		&\mathcal{Y}^k&=\mathcal{X}^k+\dfrac{t_{k-1}-1}{t_k}\left(\mathcal{X}^k-\mathcal{X}^{k-1}\right). \label{eq52}
		\end{eqnarray}
	Notice that in this case,  the function $f$ involved in \eqref{eq11} is given by $f(\mathcal{X})=\dfrac{1}{2}\left\Vert \mathcal{P}_\Omega(\mathcal{X})-\mathcal{P}_\Omega(\mathcal{M})\right\Vert_F^2$  with $\nabla f(\mathcal{X})=\mathcal{P}_\Omega(\mathcal{X})-\mathcal{P}_\Omega(\mathcal{M})$. Therefore,  the Lipshitz constant $l_f$ of $\nabla f$ is equal to $1$. 
		It is clear from Theorem \ref{theo23}, that the solution of the problem \eqref{eq 50} is the t-svt of the tensor $\mathcal{G}^k$, i.e., 
		\begin{equation}\label{eq53}
		\mathcal{X}^{k+1}=\mathcal{D}_\mu\left(\mathcal{G}^k\right).
		\end{equation}
		THe following algorithm  summarizes all the steps  of the proposed method
		\begin{algorithm}[H]
			\caption{Tensor completion using tensor nuclear norm by PGA.}
			\label{alg:5}
			\begin{algorithmic}[1]
				\STATE \textbf{Inputs:} $\mathcal{M}\in \mathbb{K}^{I_1\times I_2}_{I_3\times \dots \times I_N}$, $tol$, $\nu$, $\mu_0$. 
				\STATE \textbf{Initialize:} $\mathcal{X}^0=\mathcal{X}^1=0$, $t_0=t_1=0$, $\bar{\mu}=\nu \mu_0$, $itermax=100$, k=1.
				\WHILE {not converged}
				\STATE Update $\mathcal{\mathcal{Y}}^k$ from \eqref{eq52}.
				\STATE Update $\mathcal{\mathcal{G}}^k$ from \eqref{eq51}.
				\STATE Update $\mathcal{X}^{k+1}$ from \eqref{eq53}.
				\STATE $t_{k+1}=\dfrac{1+\sqrt{4t_k^2+1}}{2}$.
				\STATE $\mu_{k+1}=max\left(\nu \mu_k,\, \bar{\mu}\right)$.
				\ENDWHILE
			\end{algorithmic}
		\end{algorithm}
	\end{section}
\noindent In  Table \ref{tablefc}, we give the cost of the different tensor operations ( $*_\mathcal{L}$-product, $*_\mathcal{L}$-svd and $*_\mathcal{L}$-svt) when using FFT or DCT.
\begin{figure}[h!]
	\centering
	\small\addtolength{\tabcolsep}{-2pt}
	\begin{tabular}{l|l}
		\hline \multicolumn{2}{l}{\hspace*{2.5 cm}$*_\mathcal{L}$-product of $\mathcal{X}\in \mathbb{K}^{I_1\times n}_{I_3 \times \dots \times I_N}$ and $\mathcal{Y}\in \mathbb{K}^{n\times I_2}_{I_3 \times \dots \times I_N}$}\\
		\hline  \hspace*{2 cm} FFT & \hspace*{2 cm} DCT \\
		\hline   $O\left((PI_1n+PI_2n+PI_1I_2)\sum_{i=3}^{N}\log(I_i)\right)$   & $O\left((PI_1n+PI_2n+PI_1I_2)\sum_{i=3}^{N}\log(I_i)\right)$    \\
		$+4O\left(PI_1n^2I_2\right)$ & $+O\left(PI_1n^2I_2\right)$\\
		\hline
		\hline \multicolumn{2}{l}{\hspace*{4.5 cm}$*_\mathcal{L}$-svd of $\mathcal{X}\in \mathbb{K}^{I_1\times I_2}_{I_3 \times \dots \times I_N}$}\\
		\hline  \hspace*{2 cm} FFT & \hspace*{2 cm} DCT \\
		\hline   $O\left((2PI_1I_2+PI_2^2+PI_1^2)\sum_{i=3}^{N}\log(I_i)\right)$   & $O\left((2PI_1I_2+PI_2^2+PI_1^2)\sum_{i=3}^{N}\log(I_i)\right)$  \\
		$+2O\left(P\min(I_1^2I_2, I_1I_2^2)\right)$ & $+O\left(P\min(I_1^2I_2, I_1I_2^2)\right)$\\
		\hline
		\hline \multicolumn{2}{l}{\hspace*{4.5 cm}$*_\mathcal{L}$-svt of $\mathcal{X}\in \mathbb{K}^{I_1\times I_2}_{I_3 \times \dots \times I_N}$}\\
		\hline  \hspace*{2 cm} FFT & \hspace*{2 cm} DCT \\
		\hline   $O\left((2PI_1I_2+PI_2^2+PI_1^2)\sum_{i=3}^{N}\log(I_i)\right)$   & $O\left((2PI_1I_2+PI_2^2+PI_1^2)\sum_{i=3}^{N}\log(I_i)\right)$  \\
		$+2O\left(P\min(I_1^2I_2, I_1I_2^2)\right)+8O\left(PI_1n^2I_2\right)$ & $+O\left(P\min(I_1^2I_2, I_1I_2^2)\right)+2O\left(PI_1n^2I_2\right)$\\
		\hline
	\end{tabular}
	\captionof{table}{The cost of computing the $*_\mathcal{L}$-product, $*_\mathcal{L}$-svd and the $*_\mathcal{L}$-svt by using Fourier and cosine transforms.}\label{tablefc}
\end{figure}

\newpage 
\begin{section}{Numerical experiments}
	\label{sec5}
	In this section we test the performance of our algorithms for  high order tensor completion using  the Fourier and Cosine for thethe operator $\mathcal{L}$  and we will compare  the obtained results with  those ontained by some existing known algorithms on color videos. In Subsection \ref{ssec5.1},  the tests were performed with Matlab 2018a, on an Intel i5 laptop with 16 Go of memory, and in Subsection \ref{ssec5.2} we use codes with  Python on a machine that uses a CPU of type Intel Xeon Gold 6152 with a
		frequency from 2.1Ghz to 3.7 GHz and a GPU of type NVIDIA Tesla Pascal 40. All 
	the tests are computed using a single core. \\
	The quality of the obtained data can be computed by the peak signal-to-noise-ration (PSNR)  defined by  
	\begin{equation}
	PSNR=10 \,\log_{10} \dfrac{Max_{\mathcal{A}_{obt}}^2}{\left\Vert \mathcal{A}_{obt}-\mathcal{A}_{ori}\right\Vert_F^2},
	\end{equation}
	and the relative squared error (RSE) given by
	\begin{equation}
	RSE=\dfrac{\left\Vert \mathcal{A}_{ori}-\mathcal{A}_{obt}\right\Vert_F^2}{\left\Vert \mathcal{A}_{obt}\right\Vert_F^2},
	\end{equation}
	where $\mathcal{A}_{ori}$ is the original tensor, $\mathcal{A}_{obt}$ is the obtained recovered  tensor and $Max_{\mathcal{A}_{obt}}$ is the maximum pixel of the recovered tensor. The quality of the recovered data is good when the value of RSE is small and the value of PSNR is high. 
	In our  experiments of tensor completion,  we use   $M_i=F_i$ for Fourier transform  and $M_i=C_i$ for Cosine transform, where the transformed data $\mathcal{L}(\mathcal{X})$ for an $N^{th}$-order tensor is given by ${\tt fft}(\mathcal{X}, [\,],3)$ and ${\tt dct}(\mathcal{X}, [\,],3)$, respectivly. The parameters  $\mu_0$ and $\nu$ stated the  algorithm of completion have to be fixed. We set $\nu=0.9$ and $\mu_0= \nu \left\Vert \mathcal{M}\right\Vert_F$, where the stopping criterion convergence of this algorithm is as follows
	\begin{equation}
		\dfrac{\left\Vert \mathcal{Y}^k -\mathcal{X}^{k+1}\right\Vert_F}{\left\Vert \mathcal{X}^{k+1}\right\Vert_F}\leq 10^{-4}.
	\end{equation}

\medskip

	\noindent Figure \ref{fig3} shows  the data tests used  in our experiments. 
\noindent The video of xylophone is available from  Matlab, and the videos Akiyo and News are available from \footnote{url{http://trace.eas.asu.edu/yuv/}}.
	\begin{figure}
		\centering
		\begin{tabular}{ccc}
			\includegraphics[width=0.215\linewidth]{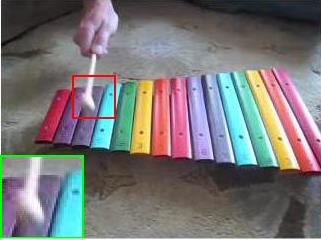}&
			\includegraphics[width=0.2\linewidth]{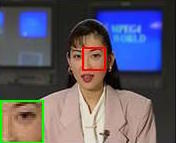}&
			\includegraphics[width=0.2\linewidth]{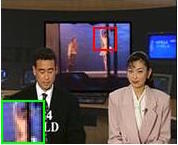}\\
			\textbf{xylophone} & \textbf{Akiyo}  &\textbf{news} 
		\end{tabular}
		\captionof{figure}{The $20^{th}$ bound of the color videos of the test data.}
		\label{fig3}
	\end{figure}
In Table \ref{tabdata}, we give the size of the different used color videos tests.
\begin{figure}[h!]
	\centering
	\begin{tabular}{l|l}
		\hline  name & size\\
		\hline   xylophone &  $240  \times 320  \times   3  \times 141$  \\
		 Akiyo 
		&  $144 \times  176  \times   3 \times  300
		$  \\
		 news & $144 \times  176  \times   3  \times 300$  \\
		\hline
	\end{tabular}
	\captionof{table}{The name of the test data (color videos) and their sizes.}\label{tabdata}
\end{figure}

	\begin{subsection}{Tensor completion for color videos}
		\label{ssec5.1}
			In this part we show the obtained results of our algorithms TNN-PGA-F and TNN-PGA-C on fourth-order tensors (color videos) and compare them   with other  ones such as the methods named  Tmac \cite{xu2013parallel} and HaLRTC \cite{liu2012tensor}. The comparison will be  in terms of the efficiency and executing times by comparing the values of RSE and PSNR, the number of iterations the required cpu-time.  In Figure \ref{fig8} we show the $20^{th}$ bound of the color videos xylophone, Akiyo and news with only  $5\%$ of the original data. In Figure \ref{fig9},  we show the $20^{th}$ bound of the recovered data  for each color video obtained by by the algorithms HaLRTC, Tmac, TNN-PGA-F and TNN-PGA-C.
			\begin{figure}[h!]
				\centering
				\begin{tabular}{ccc}
					\includegraphics[width=0.2\linewidth]{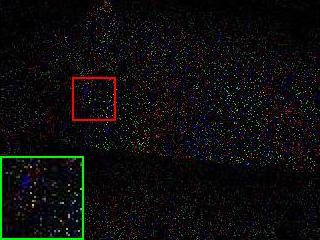}&
					\includegraphics[width=0.184\linewidth]{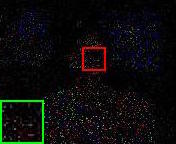}&
					\includegraphics[width=0.185\linewidth]{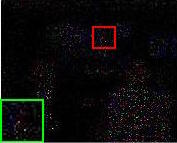}\\
					\textbf{xylophone} & \textbf{akiyo}  & \textbf{news} 
				\end{tabular}
				\captionof{figure}{The $20^{th}$ bound  of color videos of the test data with only  $5\%$ of the original data, i.e., $sr=0.05$.}
				\label{fig8}
			\end{figure}
			\begin{figure}
				\centering
				\begin{tabular}{cccc}
					\includegraphics[width=0.2\linewidth]{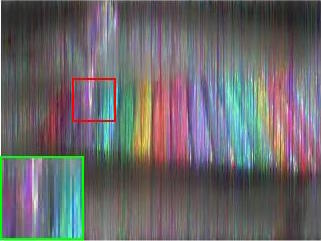}&
					\includegraphics[width=0.2\linewidth]{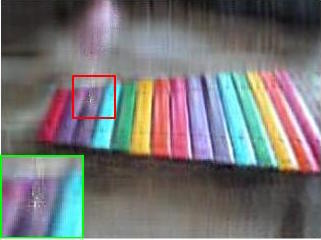}&
					\includegraphics[width=0.2\linewidth]{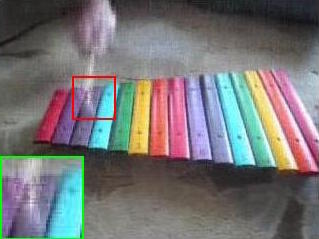}&
					\includegraphics[width=0.2\linewidth]{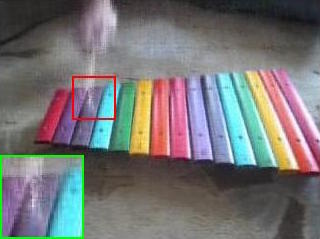}\\
				\end{tabular}
				\centering
				\begin{tabular}{cccc}
					\includegraphics[width=0.2\linewidth]{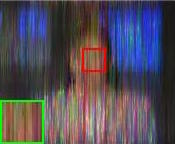}&
					\includegraphics[width=0.2\linewidth]{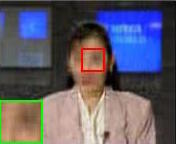}&
					\includegraphics[width=0.2\linewidth]{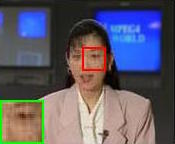}&
					\includegraphics[width=0.2\linewidth]{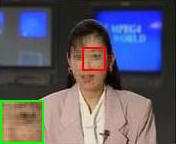}\\
				\end{tabular}
				\centering
				\begin{tabular}{cccc}
					\includegraphics[width=0.2\linewidth]{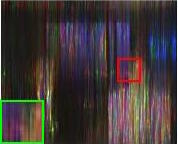}&
					\includegraphics[width=0.2\linewidth]{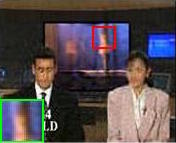}&
					\includegraphics[width=0.2\linewidth]{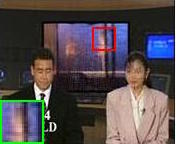}&
					\includegraphics[width=0.2\linewidth]{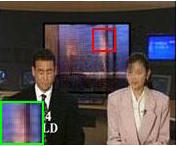}\\
					\textbf{HaLRTC} & \textbf{TMac}  & \textbf{TNN-PGA-F} &  \textbf{TNN-PGA-C}
				\end{tabular}
				\captionof{figure}{The $20^{th}$ bound of the recovered color videos obtained from the algorithms HaLRTC, Tmac, TNN-PGA-F and TNN-PGA-C for $sr=0.05$.}
				\label{fig9}
			\end{figure}
		
		\medskip
		\noindent 	In Table \ref{tab4} we report  the  values of RSE, PSNR, the number of iterations and the CPU time required by the algorithms HaLRTC, Tmac,TNN-PGA-F and TNN-PGA-C. Figure \ref{fig10} shows the curves representing the evolution of the RSE and PSNR versus the iteration number for Akiyo-video with $sr=0.05$. In Figure \ref{fig12} we give the values of the first $1000$ pixels of the recovered data of Akiyo-video for $sr=0.05$. Figure \ref{fig13}, shows  the values of RSE and PSNR for each bound of the  xylophone-video  for $sr=0.1$.
			\begin{figure}[H]
				\centering
				\small\addtolength{\tabcolsep}{-4.5pt}
				\begin{tabular}{l|l|l|l|l|l|l|l|l|l}
					\hline \multicolumn{2}{l|}{sr}& \multicolumn{4}{|l|}{$0.05$} & 	\multicolumn{4}{|l}{$0.1$}\\
					\hline Video & Algorithm & RSE & PSNR & Iteration & time & RSE & PSNR & Iteration & time \\
					\hline \multirow{4}{*}{Akiyo} & HaLRTC  & 0.3741  & 15.49  & 163  & 2464.8  & 0.2409 & 19.31  &  \;\,94 & 1632.4   \\
					&Tmca     &    0.0879
					&  30.25  & 343  &  \;\,991.0 &   0.0824  &  30.91  &  210 & \;\,587.6 \\
					&   TNN-PGA-F  & 0.0497   &33.84   & \;\,73  &\;\,619.8     &  0.0356  & 36.65  & \;\,70 & \;\,640.8\\
					&  TNN-PGA-C   & \textbf{0.0496}   & \textbf{33.91}  & \;\,99  & \textbf{\;571.2}    & \textbf{0.0350}   &  \textbf{36.80} & \;\,76 & \textbf{\;475.1} \\
					\hline \multirow{4}{*}{xylophone} &   HaLRTC   & 0.2732   & 17.49  &  163 &  3247.3   &  0.1910  & 20.60   & \;\,96  & 2012.0 \\
					&Tmca     & 0.1148   &   26.96  & 376  & 3354.7     &   0.1066 & 27.43 &  347 &  1703.7\\
					&   TNN-PGA-F  & \textbf{0.0888}   & \textbf{28.87}  &  \;\,73 &  1054.2   & 0.0680   & 30.77  & \;\,73 & 1041.4 \\
					&  TNN-PGA-C   &  0.0889  & 28.38  &  \;\,90 & \textbf{\;906.0}    &  \textbf{0.0660}  & \textbf{30.90}  & \;\,79 & \textbf{\;760.9}\\
					\hline \multirow{4}{*}{news} &   HaLRTC   &  0.4951  &  14.46 &  \;\,142 & 1876.3&    0.3694 &  17.01  & \;\,89  & 1038.2  \\
					&Tmca     &  0.1344
					& 27.95   &  1157 &  4480.0   &  0.1171  &  29.36 & 415 & 1752.6 \\
					&   TNN-PGA-F  &  \textbf{0.1058}  & \textbf{29.07}  &  \;\,\;\,71 & \;\,621.9     &  \textbf{0.0796}  &  \textbf{31.53} &\;\,69 &  \;\,573.4\\
					&  TNN-PGA-C   &  0.1156  &28.20   & \;\,\;\,99  & \textbf{\;614.0}     & 0.0838
					&  31.15 & \;\,77 & \textbf{\;475.5} \\
					\hline
				\end{tabular}
				\captionof{table}{The values of RSE, PSNR, the number of iteration and the time required by the algorithms HaLRTC, Tmac, TNN-PGA-F and TNN-PGA-C for $sr=0.05$ and $sr=0.1$.}
				\label{tab4}
			\end{figure}
			\begin{figure}[h!]
				\centering
				\hspace*{-1 cm}	\begin{tabular}{cc}
					\includegraphics[width=2.5in, height=1.5in]{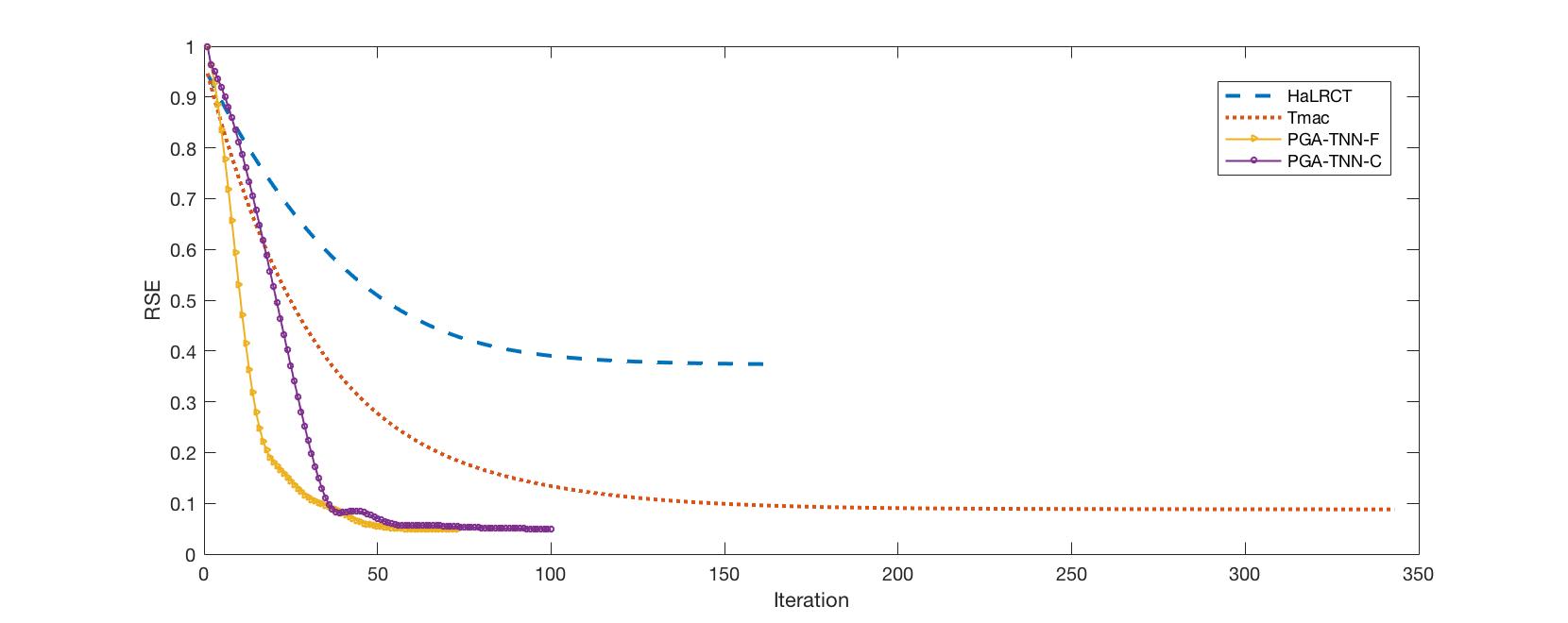}&
					\includegraphics[width=2.5in, height=1.5in]{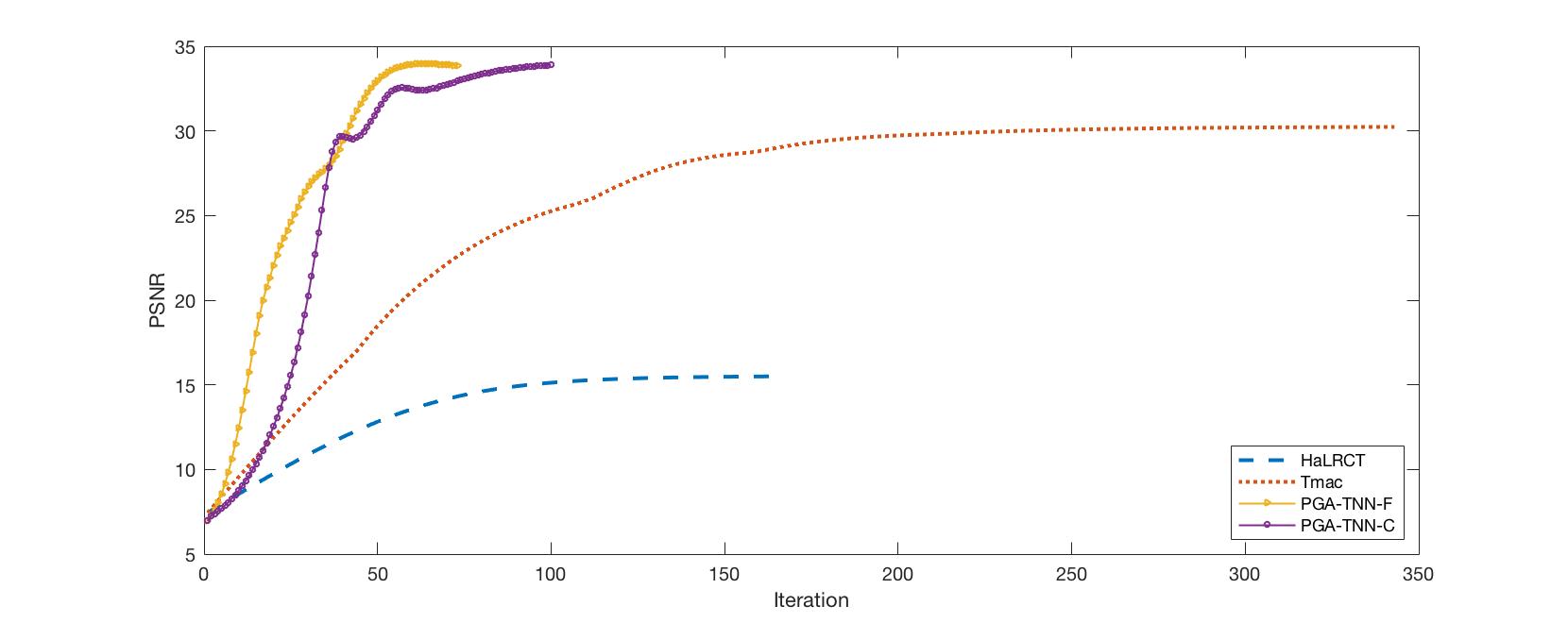}
				\end{tabular}
				\captionof{figure}{The evolution of the RSE and the PSNR values on each iteration of the algorithms HaLRTC, Tmac, TNN-PGA-F and TNN-PGA-C for the video of Akiyo with $sr=0.05$.}
				\label{fig10}
			\end{figure}
			\begin{figure}[h!]
				\centering
				\hspace*{-1 cm}	\begin{tabular}{c}
					\includegraphics[width=1\linewidth]{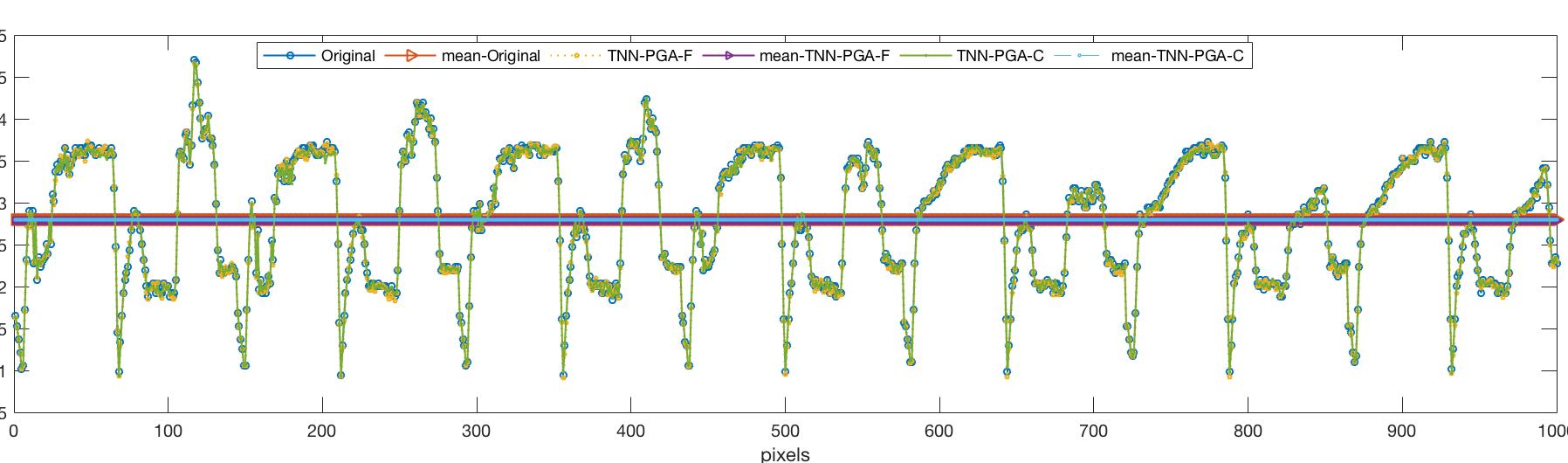}
				\end{tabular}
				\captionof{figure}{Comparison of the first $1000^{th}$ pixels of the video Akiyo obtained by the Algorithms TNN-PGA-F and TNN-PGA-F with $sr=0.05$ by the original data of the video.}
				\label{fig12}
			\end{figure}
			
			\begin{figure}[h!]
				\centering
				\hspace*{-1 cm}	\begin{tabular}{cc}
					\includegraphics[width=0.5\linewidth]{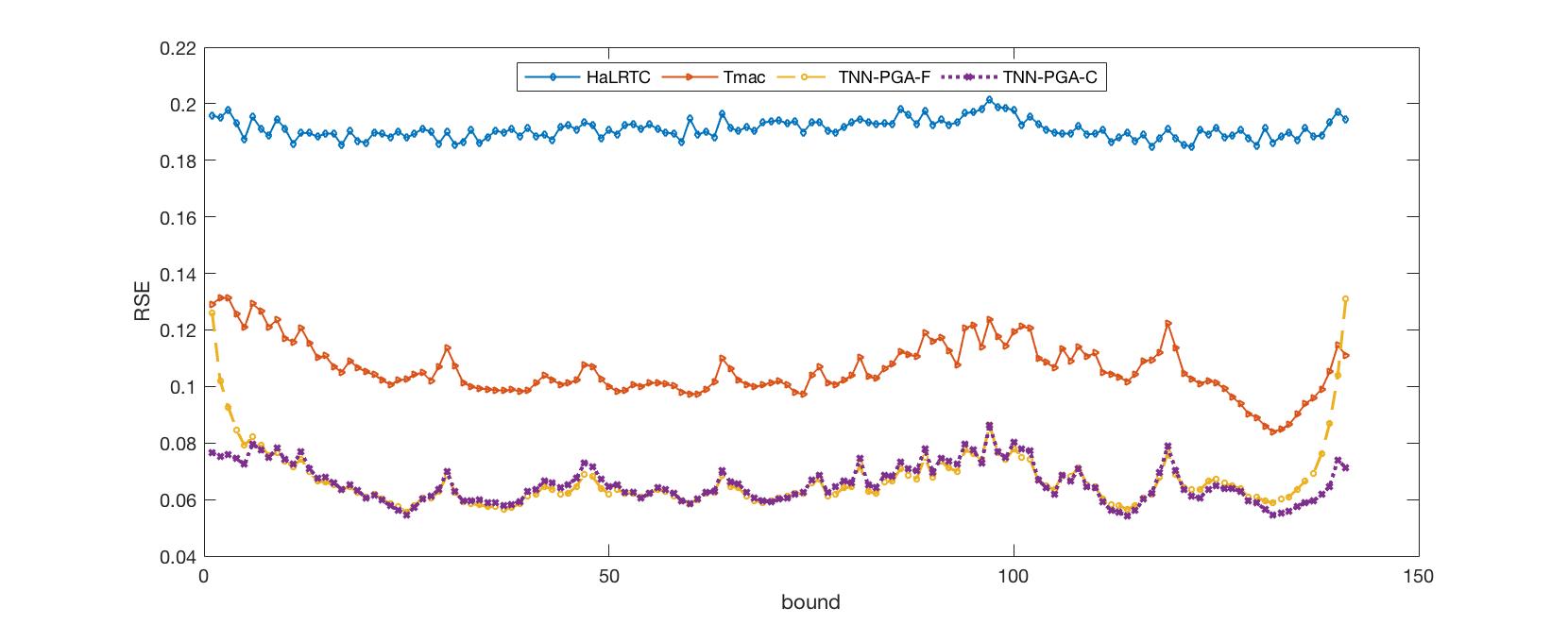}&
					\includegraphics[width=0.5\linewidth]{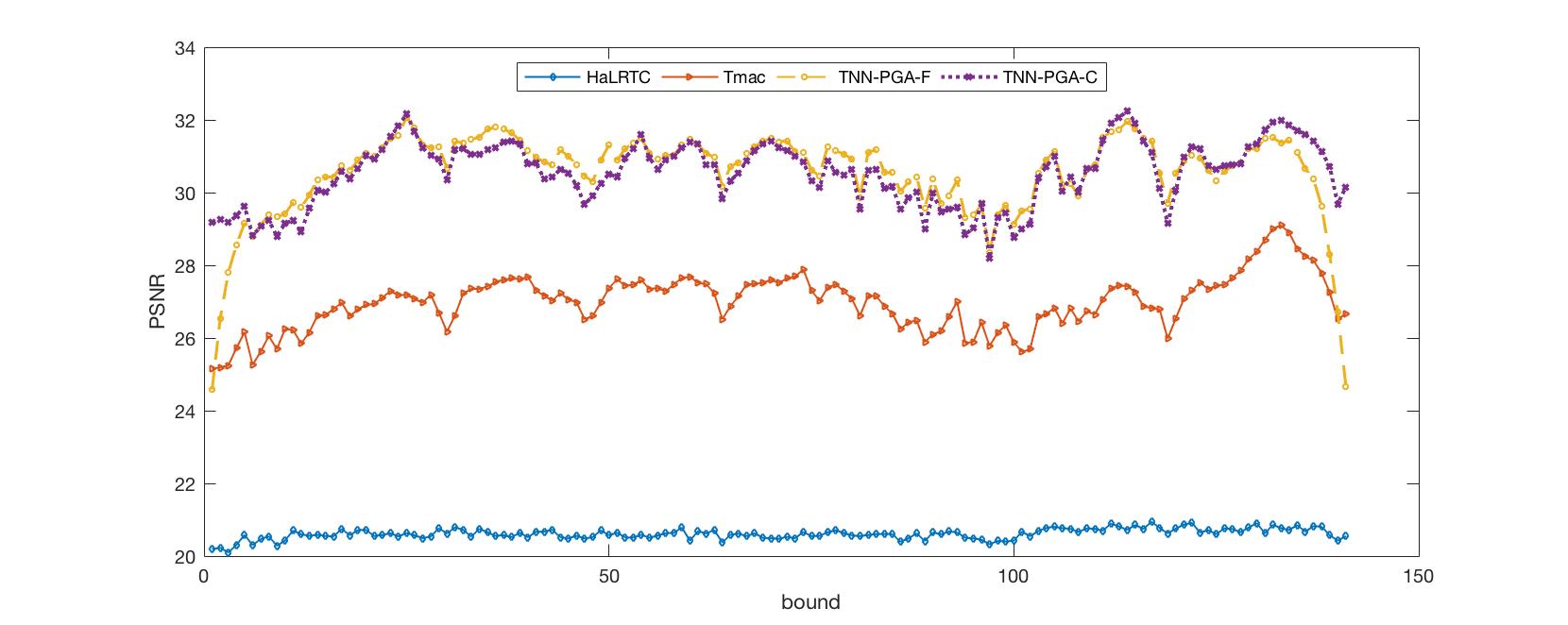}
				\end{tabular}
				\captionof{figure}{The values of RSE and PSNR ones each bound of the recovered video xylophone for $sr=0.1$.}
				\label{fig13}
			\end{figure}
		
		\medskip 
			\noindent From Figures \ref{fig9} we can see that our algorithms return very good results. Table \ref{tab4} confirms this fact  showing    an advantage for the cosine transform as compared to  the Fourier transform.
			 Figure \ref{fig10} shows  that  
				 the curves  obtained by our algorithms decreases (for RSE)  and increases (for PSNR)  quickly towards the minimum and the maximum value, respectively.  Figure \ref{fig12} shows that the values of the recovred data obtained by  the proposed two  algorithms are very close to the original data. Figure \ref{fig13} shows the efficiency of our algorithms as compared to  other ones for each bound of the used video. 
	\end{subsection}
\end{section}

\newpage
\begin{subsection}{Porting the python code for information completion to GPU using CuPy}
		\label{ssec5.2}
	Parallel computation can be very important in high performance computing  due to the limit of the use of a single core. As   a consequence,  the latest CPU manufacturers compete to have the most cores on a single CPU. For  highly parallel problems that can benefit from  more  cores  using a GPU (Graphics Processing Unit), it  is the best approach as  GPU sacrifices memory for more cores per unit and this leads to having a massively parallel capabilities over CPU. The mainly use of a GPU is to rapidly manipulate and alter memory to accelerate the creation of images.
	\\
	In our tests,  all the GPU accelerated libraries  utilize CUDA toolkit libraries which is a parallel computing platform and an API that allows interaction with a GPU in order to perform general purpose processing. That goes beyond just image data manipulation allowing  a more general approach called GP-GPU which stands for general purpose computing on graphics processing units. We used Cupy to accelerate some part of the CPU code.  CuPy is an open-source library with NumPy syntax that increases speed by doing matrix operations on NVIDIA GPUs. It is accelerated with the CUDA platform from NVIDIA and also uses CUDA-related libraries, including cuBLAS, cuDNN, cuRAND, cuSOLVER,cuSPARSE, and NCCL, to make full use of the GPU architecture.
	CuPy’s interface is highly compatible with NumPy and in most cases
	it can be used as a drop-in replacement that can easily integrated in already existing CPU code to boost the performance without much code changes.
	
	\subsubsection{Porting the python code to GPU using CuPy}
	Cupy provides an easy way to port a Python code   using Numpy and Scipy by  accelerating them  using GPU. The porting process can be  simple  by replacing some Numpy/Scipy functions by their equivalent in Cupy. \\
	
	\medskip
 \textbf{Code of tsvt by fft and dct:}
 \medskip
	\begin{itemize}
	\item \textbf{Before using CuPy}\\
	\includegraphics[scale=0.45]{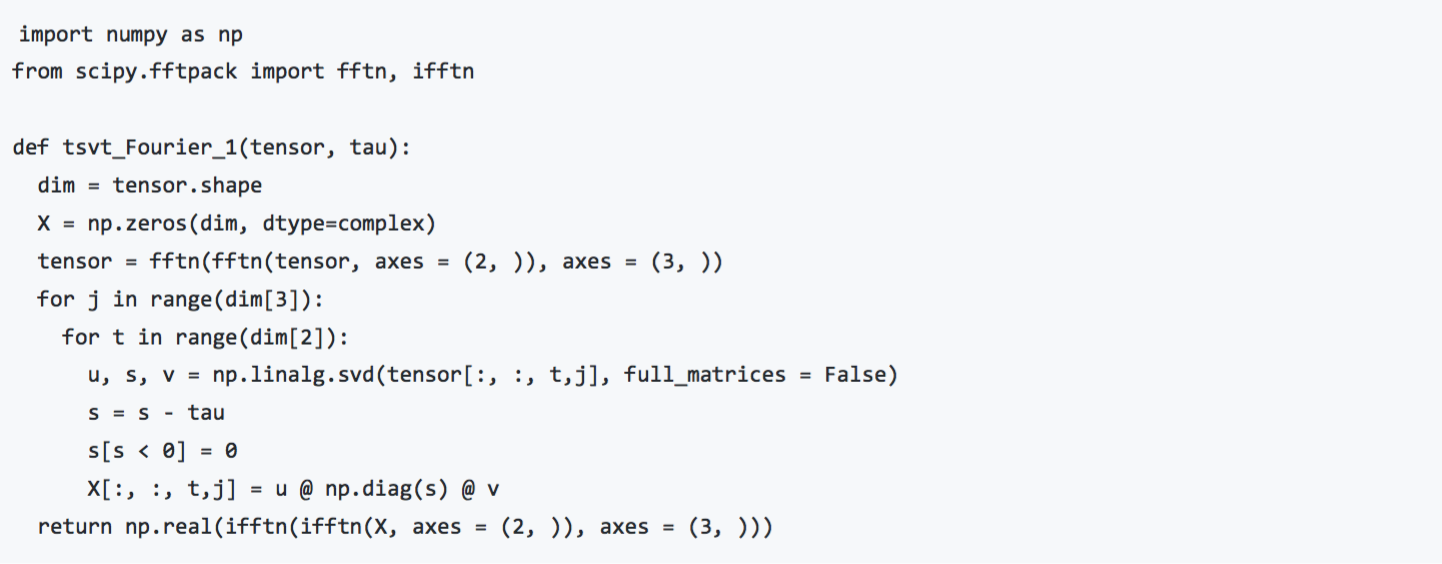}\\
	\vskip4cm
	\newpage
		\item \textbf{After using Cupy}\\
	\includegraphics[scale=0.45]{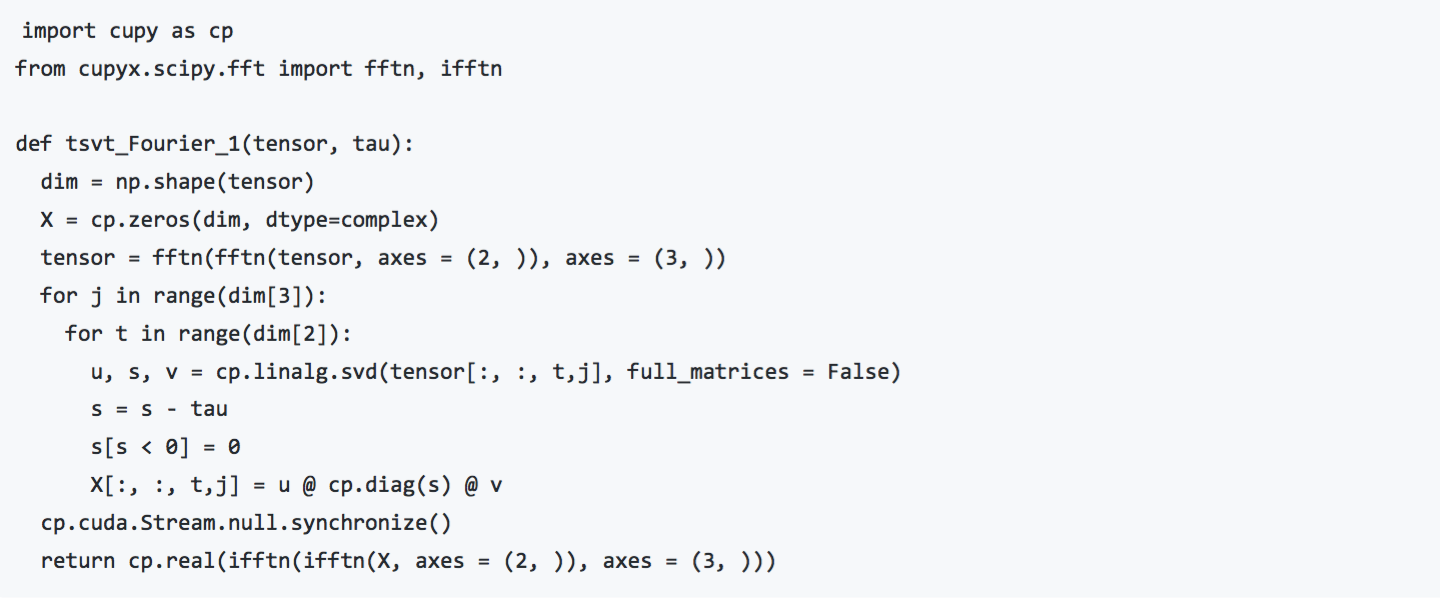}\\
		\includegraphics[scale=0.45]{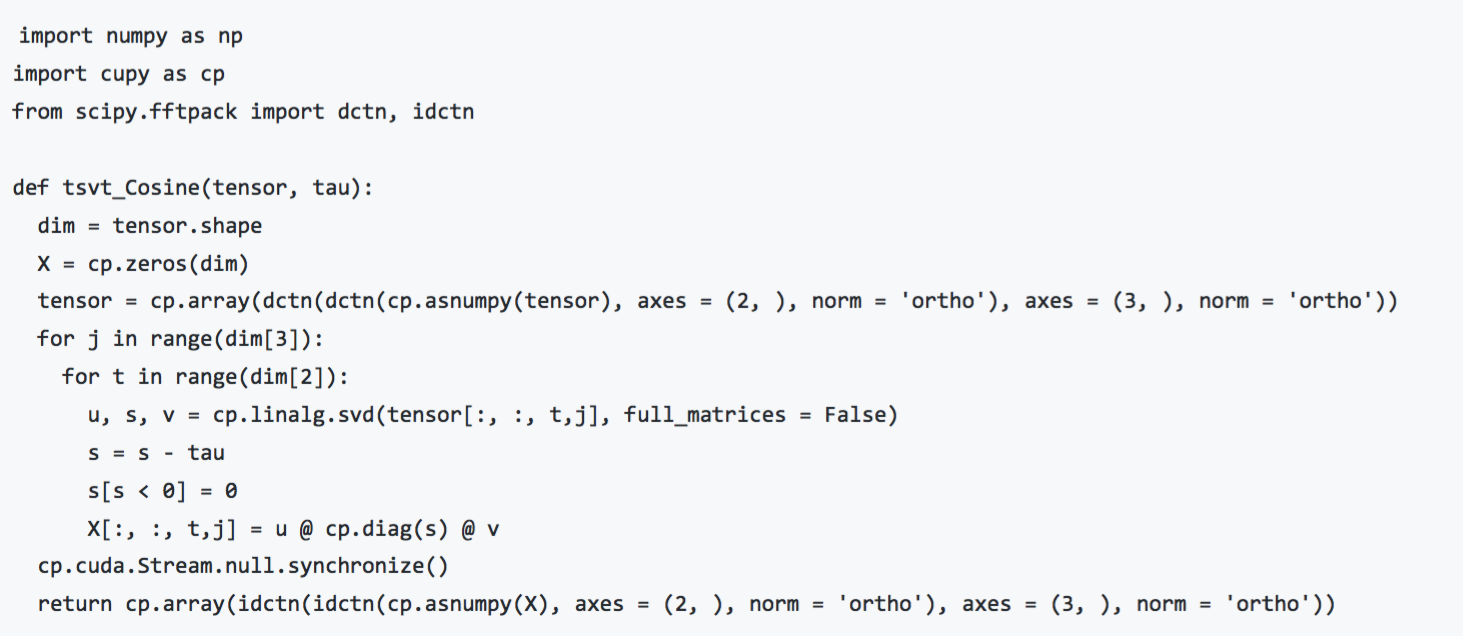}\\
		 \end{itemize}
	For the case of of tsvt-cosine by CPU we will use the same steps by changing  fftn and ifftn by dctn and idctn,  respectively
	This code provides a good example on how we can introduce Cupy to an already existing peace of code by identifying  the heavy work functions and see if they have a Cupy equivalent. The most time consuming functions are svd, fftn and ifftn. We can see in the code below that    cupy provides a GPU-accelerated implementation of those functions .
For  the algorithm that uses the DCT function instead of the FFT, the problem was more difficult since Cupy don't support the parallelizable version.\\
	\subsubsection{Numerical experiments of the problem of completion}
	In Table \ref{tab 9} we give  the size of all the data test used in our experiments (color videos: fourth-order tensors).
	\begin{figure}[h!]
		\centering
		\begin{tabular}{l|l}
			\hline Name & size \\
			\hline xylophone &   $240\times320\times3\times30$   \\
			\hline car &  $1920\times1080\times3\times30$  \\
			\hline Mgrass   & $2160\times 4096\times 3\times30$  \\
			\hline notes &      $3840\times 2160\times 3\times30$ \\
			\hline
		\end{tabular}
		\captionof{table}{The size of all the data used in the experiments.}
		\label{tab 9}
	\end{figure}
\\
		In this part we give the results of our codes of completion (TNN-PGA-F and TNN-PGA-C) by using CPU and GPU computation. In the next we denote by  PGA-F and PGA-C the codes using CPU  and by PGA-F-GPU and PGA-C-GPU those using GPU.\\ 
		In Figures \ref{fig 7.1} and \ref{fig 7.2}, we compare the evolution of the RSE and the error during the execution of the codes by CPU and GPU for two different videos 'mglass' and 'notes' with  two values of $sr$, $sr=0.05$ and $sr=0.1$. In Figure \ref{fig 7.4} we give  an  histogramme representing the required time of  PGA-F-CPU, PGA-F-GPU, PGA-C-CPU and PGA-C-GPU for $sr=0.1$.
		\begin{figure}[h!]
			\centering
			\begin{tabular}{cc}
				\includegraphics[width=0.47\linewidth]{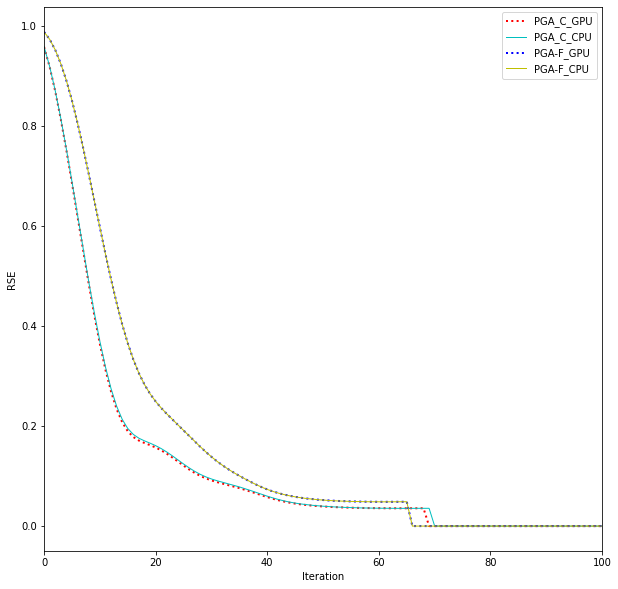}&
				\includegraphics[width=0.47\linewidth]{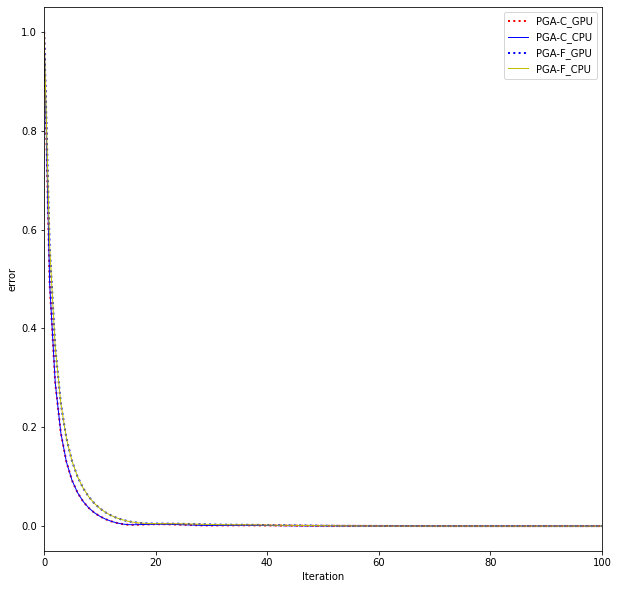}
			\end{tabular}
		\captionof{figure}{The evolution of RSE and the error on each iteration for mglass starting with $5\%$ of the original data} \label{fig 7.1}
		\end{figure}
		\begin{figure}[h!]
			\centering
			\begin{tabular}{cc}
				\includegraphics[width=0.47\linewidth]{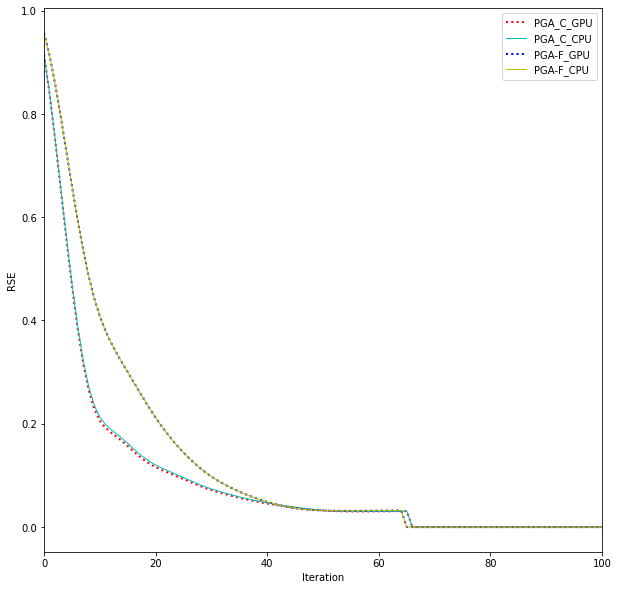}&
				\includegraphics[width=0.47\linewidth]{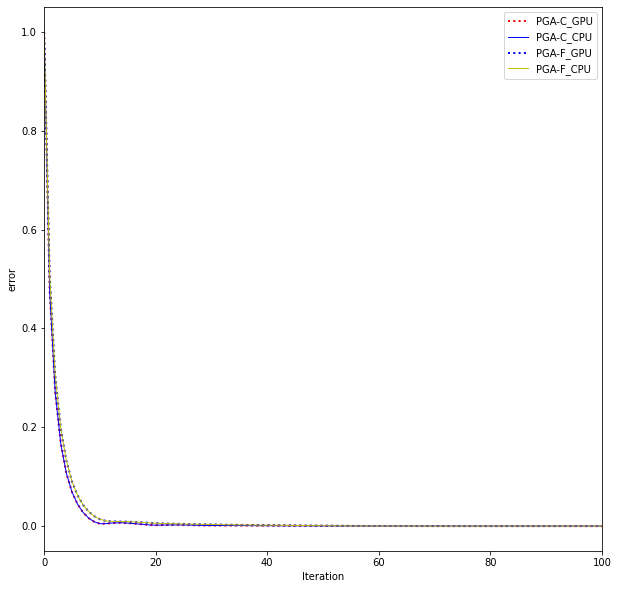}
			\end{tabular}
			\captionof{figure}{The evolution of RSE and the error on each iteration for notes starting with $10\%$ of the original data} \label{fig 7.2}
		\end{figure}
	
\noindent Figures \ref{fig 7.1} and \ref{fig 7.2} show that the RSE and the  error does not change when using the Cupy function as they are almost identical to the Numpy and Scipy ones.
\begin{figure}[h!]
	\centering
	\begin{tabular}{c}
		\includegraphics[width=0.9\linewidth]{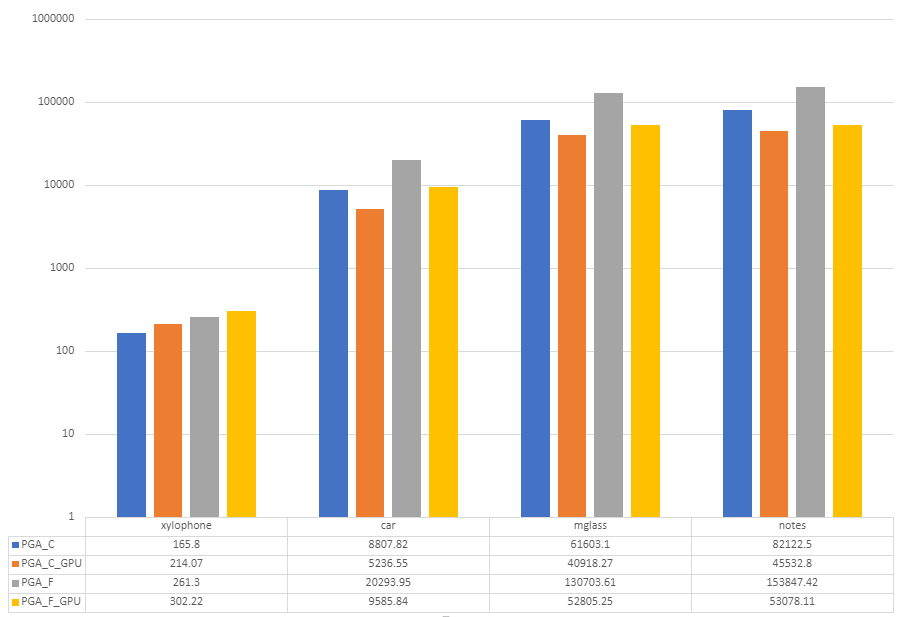}
	\end{tabular}
	\captionof{figure}{time comparison for data starting with $10\%$ of the original data.}\label{fig 7.4}
\end{figure}
\\
In Table \ref{tab 7.2} we   show the speed-up between the sequential and parallel computations. We reported the average $\tau$ defined by   using the following formula
\begin{equation}\label{tau}
	\tau= \dfrac{cpu\_time}{gpu\_time}.
\end{equation}
\begin{figure}[h!]
	\centering
	\begin{tabular}{l|l|l|l|l}
		\hline & xylohone & car & mglass & notes \\
		\hline PGA-C & $0.77$ & $1.68$ & $1.5$ & $1.80$\\
		\hline PGA-F & $0.86$ & $2.11$ & $2.47$ & $2.89$\\
		\hline
	\end{tabular}
\captionof{table}{Speed up percentage for the data set that started from $10\%$ from the original data.}\label{tab 7.2}
\end{figure}
As shown in Table \ref{tab 7.2}, the obtained speed-up values show how much we can boost the performance of our  codes by using the GPU. When the data set is small, as in the case for xylophone, there is no need to use GPU becuase in that case the returned cpu-time is smaller than the one optained by GPU. This  performance can be explained  by the fact that for small problems,  the transfer of the data in parallel computation, requires a  significant time compared the classical computation for which no need of transfert data is needed. 
When the data becomes larger and the targeted accelerated function is taking a significant time from the total runtime we can see a big speed up using the GPU accelerated functions as kernels.\\
\end{subsection}
\newpage
\section{Conclusion}
In this paper we presented a new tensor-tensor product for high orders. Using this product, we defined a new high -order SVD and some related properties. We gave some theoretical results for the tensor product. We used  this   tensor product for  tensor completion using the proximal gradient algorithm.  In the numerical section, we showed some test on color-videos and   used GPU computation to get fast  computation . The presented numerical experiments show the efficiency of our proposed algorithms.

\end{document}